\documentclass[10pt]{amsart}

\usepackage{amsfonts}
\usepackage{amsmath} 
\usepackage{latexsym}
\usepackage{amssymb}
\usepackage{verbatim}
\usepackage[usenames]{color}
\usepackage{hyperref}
\usepackage{url}
\usepackage{mathrsfs}
\usepackage{tikz,tikz-qtree,ifthen}

\begin{document}

\newcommand{\s}{\sigma}
\newcommand{\al}{\alpha}
\newcommand{\om}{\omega}
\newcommand{\be}{\beta}
\newcommand{\la}{\lambda}
\newcommand{\vp}{\varphi}

\newcommand{\bo}{\mathbf{0}}
\newcommand{\bone}{\mathbf{1}}

\newcommand{\sse}{\subseteq}
\newcommand{\contains}{\supseteq}
\newcommand{\forces}{\Vdash}

\newcommand{\FIN}{\mathrm{FIN}}
\newcommand{\Fin}{\mathrm{Fin}}
\newcommand{\fin}{\mathrm{fin}}

\newcommand{\ve}{\vee}
\newcommand{\w}{\wedge}
\newcommand{\bv}{\bigvee}
\newcommand{\bw}{\bigwedge}
\newcommand{\bcup}{\bigcup}
\newcommand{\bcap}{\bigcap}

\newcommand{\rgl}{\rangle}
\newcommand{\lgl}{\langle}
\newcommand{\lr}{\langle\ \rangle}
\newcommand{\re}{\restriction}

\newcommand{\bB}{\mathbb{B}}
\newcommand{\bP}{\mathbb{P}}
\newcommand{\bR}{\mathbb{R}}
\newcommand{\bW}{\mathbb{W}}
\newcommand{\bX}{\mathbb{X}}
\newcommand{\bN}{\mathbb{N}}
\newcommand{\bQ}{\mathbb{Q}}
\newcommand{\bS}{\mathbb{S}}
\newcommand{\St}{\tilde{S}}
   
\newcommand{\sd}{\triangle}
\newcommand{\cl}{\prec}
\newcommand{\cle}{\preccurlyeq}
\newcommand{\cg}{\succ}
\newcommand{\cge}{\succcurlyeq}
\newcommand{\dom}{\mathrm{dom}}
\newcommand{\ran}{\mathrm{ran}}

\newcommand{\lra}{\leftrightarrow}
\newcommand{\ra}{\rightarrow}
\newcommand{\llra}{\longleftrightarrow}
\newcommand{\Lla}{\Longleftarrow}
\newcommand{\Lra}{\Longrightarrow}
\newcommand{\Llra}{\Longleftrightarrow}
\newcommand{\rla}{\leftrightarrow}
\newcommand{\lora}{\longrightarrow}
\newcommand{\E}{\mathrm{E}}
\newcommand{\rank}{\mathrm{rank}}
\newcommand{\lefin}{\le_{\mathrm{fin}}}
\newcommand{\Ext}{\mathrm{Ext}}
\newcommand{\lelex}{\le_{\mathrm{lex}}}
\newcommand{\depth}{\mathrm{depth}}

\newcommand{\Erdos}{Erd{\H{o}}s}
\newcommand{\Pudlak}{Pudl{\'{a}}k}
\newcommand{\Rodl}{R{\"{o}}dl}
\newcommand{\Proml}{Pr{\"{o}}ml}
\newcommand{\Fraisse}{Fra{\"{i}}ss{\'{e}}}
\newcommand{\Sokic}{Soki{\'{c}}}
\newcommand{\Nesetril}{Ne{\v{s}}et{\v{r}}il}

\newtheorem{thm}{Theorem}  
\newtheorem{prop}[thm]{Proposition} 
\newtheorem{lem}[thm]{Lemma} 
\newtheorem{cor}[thm]{Corollary} 
\newtheorem{fact}[thm]{Fact}    
\newtheorem{facts}[thm]{Facts}      
\newtheorem*{thmMT}{Main Theorem}
\newtheorem*{thmMTUT}{Main Theorem for $\vec{\mathcal{U}}$-trees}
\newtheorem*{thmnonumber}{Theorem}
\newtheorem*{mainclaim}{Main Claim}
\newtheorem{claim}{Claim}
\newtheorem*{claim1}{Claim $1$}
\newtheorem*{claim2}{Claim $2$}
\newtheorem*{claim3}{Claim $3$}
\newtheorem*{claim4}{Claim $4$}

\theoremstyle{definition}   
\newtheorem{defn}[thm]{Definition} 
\newtheorem{example}[thm]{Example} 
\newtheorem{conj}[thm]{Conjecture} 
\newtheorem{prob}[thm]{Problem} 
\newtheorem{examples}[thm]{Examples}
\newtheorem{question}[thm]{Question}
\newtheorem{problem}[thm]{Problem}
\newtheorem{openproblems}[thm]{Open Problems}
\newtheorem{openproblem}[thm]{Open Problem}
\newtheorem{conjecture}[thm]{Conjecture}
\newtheorem*{problem1}{Problem 1}
\newtheorem*{problem2}{Problem 2}
\newtheorem*{problem3}{Problem 3}
\newtheorem*{notation}{Notation}

\theoremstyle{remark} 
\newtheorem*{rem}{Remark} 
\newtheorem*{rems}{Remarks} 
\newtheorem*{ack}{Acknowledgments} 
\newtheorem*{note}{Note}
\newtheorem{claimn}{Claim}
\newtheorem{subclaim}{Subclaim}
\newtheorem*{subclaimnn}{Subclaim}
\newtheorem*{subclaim1}{Subclaim (i)}
\newtheorem*{subclaim2}{Subclaim (ii)}
\newtheorem*{subclaim3}{Subclaim (iii)}
\newtheorem*{subclaim4}{Subclaim (iv)}
\newtheorem{case}{Case}
 \newtheorem*{case1}{Case 1}
\newtheorem*{case2}{Case 2}
\newtheorem*{case3}{Case 3}
\newtheorem*{case4}{Case 4}

\newcommand{\noprint}[1]{\relax}
\newenvironment{nd}{\noindent\color{red}ND: }{}
\newenvironment{jsm}{\noindent\color{blue}SM: }{}

\title[High dimensional Ellentuck spaces]{High dimensional  Ellentuck spaces and initial chains in the Tukey structure of non-p-points}
\author{Natasha Dobrinen}
\address{Department of Mathematics\\
  University of Denver \\
   2280 S Vine St\\ Denver, CO \ 80208 U.S.A.}
\email{natasha.dobrinen@du.edu}
\urladdr{\url{http://web.cs.du.edu/~ndobrine}} 
\thanks{This work  was partially supported by 
the Fields Institute and the National Science Foundation during the 2012 Thematic Program on Forcing and Its Applications,
 National Science Foundation Grant DMS-1301665,  and   Simons Foundation Collaboration Grant 245286}

\begin{abstract}
The generic ultrafilter $\mathcal{G}_2$ forced by 
$\mathcal{P}(\om\times\om)/(\Fin\otimes\Fin)$ was recently proved to be  neither maximum nor minimum in the Tukey order of ultrafilters (\cite{Blass/Dobrinen/Raghavan13}), but it was  left open where exactly  in the Tukey order it lies.
We prove that $\mathcal{G}_2$ is  in fact  Tukey minimal over its  projected Ramsey ultrafilter.
Furthermore, we 
prove that for each $k\ge 2$, the collection of all nonprincipal ultrafilters Tukey reducible to the generic  ultrafilter  $\mathcal{G}_k$ forced by $\mathcal{P}(\om^k)/\Fin^{\otimes k}$ forms a  chain  of length $k$.
Essential to the proof is the extraction of 
 a dense subset $\mathcal{E}_k$ from  $(\Fin^{\otimes k})^+$    which we prove  to be a topological Ramsey space.
The spaces $\mathcal{E}_k$, $k\ge 2$,  form a hierarchy of  high dimensional  Ellentuck spaces.
New Ramsey-classification theorems for equivalence relations on fronts on $\mathcal{E}_k$ are proved, extending the \Pudlak-\Rodl\ Theorem for fronts on the Ellentuck space, which are applied to  find the  Tukey structure below $\mathcal{G}_k$.
\end{abstract}


\maketitle


\section{Introduction}\label{sec.intro}

The structure of the Tukey types of ultrafilters is a current  focus of  research in set theory and structural Ramsey theory; the interplay between the two areas has proven fruitful for each.
This particular line of research began in \cite{Raghavan/Todorcevic12}, in which Todorcevic showed that selective ultrafilters are minimal in the Tukey order via an insightful application of the \Pudlak-\Rodl\ Theorem canonizing equivalence relations on barriers on the Ellentuck space.
Soon after, new topological Ramsey spaces were constructed by Dobrinen and Todorcevic in \cite{Dobrinen/Todorcevic14} and \cite{Dobrinen/Todorcevic15}, in which   Ramsey-classification theorems for equivalence relations on fronts were proved and applied to find initial Tukey structures of the associated p-point ultrafilters which are decreasing chains of order-type $\al+1$ for each countable ordinal $\al$.
Recent work of Dobrinen, Mijares, and Trujillo in \cite{Dobrinen/Mijares/Trujillo14} provided a template for constructing  topological Ramsey spaces which have associated p-point ultrafilters with initial Tukey structures which are finite Boolean algebras, extending the work in \cite{Dobrinen/Todorcevic14} .

This paper is the first to examine initial Tukey structures of non-p-points.
Our work was motivated by
  \cite{Blass/Dobrinen/Raghavan13},  in which Blass, Dobrinen, and Raghavan studied 
the Tukey type of the generic ultrafilter $\mathcal{G}_2$ forced by $\mathcal{P}(\om\times\om)/\Fin\otimes\Fin$.
As this ultrafilter was known to be a Rudin-Keisler immediate successor of its projected selective ultrafilter (see Proposition 30 in \cite{Blass/Dobrinen/Raghavan13}) and at the same time be neither a p-point nor a Fubini iterate of p-points,
it became of interest to see where in the Tukey hierarchy this ultrafilter lies.

At this point, we review the  definitions and background necessary to understand the motivation for the current project.
Throughout, we consider ultrafilters to be partially ordered by reverse inclusion.
Given two ultrafilters $\mathcal{U}$ and $\mathcal{V}$,
we say that $\mathcal{V}$ is {\em Tukey reducible to} $\mathcal{U}$, and write $\mathcal{V}\le_T\mathcal{U}$,  if there is a function $f:\mathcal{U}\ra\mathcal{V}$ which maps each filter base of $\mathcal{U}$ to a filter base of $\mathcal{V}$.
$\mathcal{U}$ and $\mathcal{V}$ are {\em Tukey equivalent} if both 
$\mathcal{U}\le_T\mathcal{V}$ and $\mathcal{V}\le_T\mathcal{U}$.
In this case we write $\mathcal{U}\equiv_T\mathcal{V}$.
The Tukey equivalence class of an ultrafilter $\mathcal{U}$ is called its {\em Tukey type}.
Given an ultrafilter $\mathcal{U}$, we use the terminology {\em initial Tukey structure below $\mathcal{U}$} to denote the structure (under Tukey reducibility) of the collection of Tukey types of   all ultrafilters Tukey reducible to $\mathcal{U}$.
For ultrafilters, Tukey equivalence is the same as cofinal equivalence.

The partial order $([\mathfrak{c}]^{<\om},\sse)$ is the maximum Tukey type for all ultrafilters on a countable base set.
In \cite{Isbell65}, Isbell asked whether there is always more than one Tukey type.
The recent  surge in activity began with \cite{Milovich08} in which Milovich showed under $\lozenge$ that there can be more than one Tukey type.
This was improved in \cite{Dobrinen/Todorcevic11}, where it was shown that all p-points are strictly below the Tukey maximum.
For more background on the Tukey theory of ultrafilters, the reader is referred to the survey article \cite{Dobrinen13}.

The  paper \cite{Blass/Dobrinen/Raghavan13} of Blass, Dobrinen, and Raghavan began the investigation of the Tukey theory of the generic ultrafilter $\mathcal{G}_2$ forced by  $\mathcal{P}(\om\times\om)/\Fin\otimes\Fin$,
where $\Fin\otimes\Fin$ denotes the collection of subsets of $\om\times\om$ in which all but finitely many fibers are finite.
The motivation for this study was the open problem of whether the classes of basically generated ultrafilters and countable iterates of Fubini products of p-points are the same class of ultrafilters.
The notion of a {\em basically generated} ultrafilter was introduced by Todorcevic to extract the key property of Fubini iterates of p-points which make them strictly below $([\mathfrak{c}]^{<\om},\sse)$, the top of the Tukey hierarchy.
In Section 3 of \cite{Dobrinen/Todorcevic11}, Dobrinen and Todorcevic showed that the class of basically generated ultrafilters contains all countable iterates of Fubini products of p-points.  They then asked whether there is a basically generated ultrafilter which is not Tukey  equivalent to some iterated Fubini product of p-points. This question is still open.

Since it is well-known that the generic ultrafilter $\mathcal{G}_2$ is not a Fubini product of p-points, yet is a Rudin-Keisler immediate successor of its projected selective ultrafilter,  
Blass 
asked whether  $\mathcal{G}_2$  is Tukey maximum, and if not, then whether it is basically generated.
In \cite{Blass/Dobrinen/Raghavan13}, 
Blass proved that $\mathcal{G}_2$ is a weak p-point which
has the best partition property that a non-p-point can have. 
Dobrinen and Raghavan independently proved that $\mathcal{G}_2$ is not Tukey maximum, which was improved by 
Dobrinen  in Theorem 49 in \cite{Blass/Dobrinen/Raghavan13}  by showing that $(\mathcal{G}_2,\contains)\not\ge_T ([\om_1]^{<\om},\sse)$, thereby showing in a strong way that $\mathcal{G}_2$ does not have the maximum Tukey type for ultrafilters on a countable base set.
Answering the other question of Blass, Raghavan showed in Theorem 60 in \cite{Blass/Dobrinen/Raghavan13} that $\mathcal{G}_2$ is not basically generated.
However, that paper left open the question of where exactly in the Tukey hierarchy $\mathcal{G}_2$ lies, and what the structure of the Tukey types below it actually is.

In this paper, we prove  that  the initial Tukey structure below $\mathcal{G}_2$ is exactly a chain of order-type 2.
In particular, $\mathcal{G}_2$ is
 the immediate Tukey successor of its projected selective ultrafilter.
Extending this further, we  investigate the initial Tukey structure of the generic ultrafilters forced by $\mathcal{P}(\om^k)/\Fin^{\otimes k}$.
Here, 
$\Fin^{k+1}$ is defined recursively:
$\Fin^1$ denotes the collection of finite subsets of $\om$;  for 
 $k\ge 1$, 
 $\Fin^{\otimes k+1}$ denotes the collection of subsets $X\sse\om^{k+1}$ such that for all but finitely many $i_0\in\om$, the set $\{(i_0,j_1\dots,j_k)\in\om^{k+1}:j_1,\dots,j_k\in\om\}$ is in $\Fin^{\otimes k}$.
We prove in Theorem \ref{thm.main}
that for all $k\ge 2$,
the generic ultrafilter $\mathcal{G}_k$ 
forced by $\mathcal{P}(\om^k)/\Fin^{\otimes k}$
has initial Tukey structure (of nonprincipal ultrafilters) exactly a chain of size $k$.
We also show that the Rudin-Keisler structures below $\mathcal{G}_k$ is 
exactly a chain of size $k$. 
Thus, the Tukey structure below $\mathcal{G}_k$ mirrors the Rudin-Keisler structure below $\mathcal{G}_k$.

We remark that the structure of the spaces $\mathcal{E}_k$ provide 
a clear way of understanding the partition relations satisfied by $\mathcal{G}_k$.
In particular, our space $\mathcal{E}_2$ 
provides  an alternate  method for proving  Theorem 31 of \cite{Blass/Dobrinen/Raghavan13}, due to Blass,  where it is shown that $\mathcal{G}_2$ has the best partition properties that a non-p-point can have.

The paper is organized as follows.
Section \ref{sec.reviewtRs} provides some  background on topological Ramsey spaces from Todorcevic's book \cite{TodorcevicBK10}. 
The new topological Ramsey spaces $\mathcal{E}_k$, $k\ge 2$,  are introduced in  Section \ref{sec.defR}.
These spaces are formed by thinning the forcing $((\Fin^{\otimes k})^+,\sse^{\Fin^{\otimes k}})$, which is  forcing equivalent to $\mathcal{P}(\om^k)/\Fin^{\otimes k}$, to a dense subset and judiciously choosing the finitization map so as to form a topological Ramsey space.
Once formed, these spaces  are seen to be  high dimensional extensions  of the Ellentuck space.
The Ramsey-classification theorem generalizing the \Pudlak-\Rodl\ Theorem to all spaces $\mathcal{E}_k$, $k\ge 2$ is proved in Theorem \ref{thm.rc} of Section
\ref{sec.rct}.
Theorem \ref{thm.canon}  in Section \ref{sec.basic}
shows 
  that any monotone cofinal map from the generic ultrafilter $\mathcal{G}_k$ into some other ultrafilter is actually represented on a filter base by some monotone, end-extension preserving finitary map.
This is the analogue  of p-points having continuous cofinal maps for our current setting,
and is sufficient for the arguments using canonical maps on fronts to find the initial Tukey structure below $\mathcal{G}_k$, which we do in Theorem \ref{thm.main} of Section \ref{sec.Tukey}.

\bf Acknowledgment. \rm 
The author thanks S. Todorcevic for his suggestion of investigating  the Tukey types of the generic ultrafilters for the higher dimensional  forcings $\mathcal{P}(\om^k)/\Fin^{\otimes k}$.


\section{Basics of  general topological Ramsey spaces}\label{sec.reviewtRs}

For the reader's convenience, we  provide here a brief review of topological Ramsey spaces.
Building on earlier work of Carlson and Simpson in \cite{Carlson/Simpson90}, Todorcevic distilled the key properties of the Ellentuck space into  four axioms, \bf A.1  \rm -  \bf A.4\rm, which guarantee that a space is a topological Ramsey space.
As several recent papers have been devoted to topological Ramsey spaces, related canonical equivalence relations on fronts and their applications to initial Tukey structures of associated ultrafilters (see \cite{Dobrinen/Todorcevic14}, \cite{Dobrinen/Todorcevic15} and \cite{Dobrinen/Mijares/Trujillo14}), we reproduce  here only information necessary to aiding the reader in understanding the proofs in this paper.
For further background, we refer the reader to Chapter 5 of \cite{TodorcevicBK10}.

The  axioms \bf A.1  \rm -  \bf A.4\rm,
are defined for triples
$(\mathcal{R},\le,r)$
of objects with the following properties.
$\mathcal{R}$ is a nonempty set,
$\le$ is a quasi-ordering on $\mathcal{R}$,
 and $r:\mathcal{R}\times\om\ra\mathcal{AR}$ is a mapping giving us the sequence $(r_n(\cdot)=r(\cdot,n))$ of approximation mappings, where
$\mathcal{AR}$ is  the collection of all finite approximations to members of $\mathcal{R}$.
For $a\in\mathcal{AR}$ and $A,B\in\mathcal{R}$,
\begin{equation}
[a,B]=\{A\in\mathcal{R}:A\le B\mathrm{\ and\ }(\exists n)\ r_n(A)=a\}.
\end{equation}

For $a\in\mathcal{AR}$, let $|a|$ denote the length of the sequence $a$.  Thus, $|a|$ equals the integer $k$ for which $a=r_k(a)$.
For $a,b\in\mathcal{AR}$, $a\sqsubseteq b$ if and only if $a=r_m(b)$ for some $m\le |b|$.
$a\sqsubset b$ if and only if $a=r_m(b)$ for some $m<|b|$.
For each $n<\om$, $\mathcal{AR}_n=\{r_n(A):A\in\mathcal{R}\}$.
\vskip.1in

\begin{enumerate}
\item[\bf A.1]\rm
\begin{enumerate}
\item
$r_0(A)=\emptyset$ for all $A\in\mathcal{R}$.\vskip.05in
\item
$A\ne B$ implies $r_n(A)\ne r_n(B)$ for some $n$.\vskip.05in
\item
$r_n(A)=r_m(B)$ implies $n=m$ and $r_k(A)=r_k(B)$ for all $k<n$.\vskip.1in
\end{enumerate}
\item[\bf A.2]\rm
There is a quasi-ordering $\le_{\mathrm{fin}}$ on $\mathcal{AR}$ such that\vskip.05in
\begin{enumerate}
\item
$\{a\in\mathcal{AR}:a\le_{\mathrm{fin}} b\}$ is finite for all $b\in\mathcal{AR}$,\vskip.05in
\item
$A\le B$ iff $(\forall n)(\exists m)\ r_n(A)\le_{\mathrm{fin}} r_m(B)$,\vskip.05in
\item
$\forall a,b,c\in\mathcal{AR}[a\sqsubset b\wedge b\le_{\mathrm{fin}} c\ra\exists d\sqsubset c\ a\le_{\mathrm{fin}} d]$.\vskip.1in
\end{enumerate}
\end{enumerate}

The number $\depth_B(a)$ is the least $n$, if it exists, such that $a\le_{\mathrm{fin}}r_n(B)$.
If such an $n$ does not exist, then we write $\depth_B(a)=\infty$.
If $\depth_B(a)=n<\infty$, then $[\depth_B(a),B]$ denotes $[r_n(B),B]$.

\begin{enumerate}
\item[\bf A.3] \rm
\begin{enumerate}
\item
If $\depth_B(a)<\infty$ then $[a,A]\ne\emptyset$ for all $A\in[\depth_B(a),B]$.\vskip.05in
\item
$A\le B$ and $[a,A]\ne\emptyset$ imply that there is $A'\in[\depth_B(a),B]$ such that $\emptyset\ne[a,A']\sse[a,A]$.\vskip.1in
\end{enumerate}
\end{enumerate}

If $n>|a|$, then  $r_n[a,A]$ denotes the collection of all $b\in\mathcal{AR}_n$ such that $a\sqsubset b$ and $b\le_{\mathrm{fin}} A$.

\begin{enumerate}
\item[\bf A.4]\rm
If $\depth_B(a)<\infty$ and if $\mathcal{O}\sse\mathcal{AR}_{|a|+1}$,
then there is $A\in[\depth_B(a),B]$ such that
$r_{|a|+1}[a,A]\sse\mathcal{O}$ or $r_{|a|+1}[a,A]\sse\mathcal{O}^c$.\vskip.1in
\end{enumerate}

The  {\em Ellentuck topology} on $\mathcal{R}$ is the topology generated by the basic open sets
$[a,B]$;
it extends the usual metrizable topology on $\mathcal{R}$ when we consider $\mathcal{R}$ as a subspace of the Tychonoff cube $\mathcal{AR}^{\bN}$.
Given the Ellentuck topology on $\mathcal{R}$,
the notions of nowhere dense, and hence of meager are defined in the natural way.
We  say that a subset $\mathcal{X}$ of $\mathcal{R}$ has the {\em property of Baire} iff $\mathcal{X}=\mathcal{O}\cap\mathcal{M}$ for some Ellentuck open set $\mathcal{O}\sse\mathcal{R}$ and Ellentuck meager set $\mathcal{M}\sse\mathcal{R}$.

\begin{defn}[\cite{TodorcevicBK10}]\label{defn.5.2}
A subset $\mathcal{X}$ of $\mathcal{R}$ is {\em Ramsey} if for every $\emptyset\ne[a,A]$,
there is a $B\in[a,A]$ such that $[a,B]\sse\mathcal{X}$ or $[a,B]\cap\mathcal{X}=\emptyset$.
$\mathcal{X}\sse\mathcal{R}$ is {\em Ramsey null} if for every $\emptyset\ne [a,A]$, there is a $B\in[a,A]$ such that $[a,B]\cap\mathcal{X}=\emptyset$.

A triple $(\mathcal{R},\le,r)$ is a {\em topological Ramsey space} if every subset of $\mathcal{R}$  with the property of Baire  is Ramsey and if every meager subset of $\mathcal{R}$ is Ramsey null.
\end{defn}

The following result can be found as Theorem
5.4 in \cite{TodorcevicBK10}.

\begin{thm}[Abstract Ellentuck Theorem]\label{thm.AET}\rm \it
If $(\mathcal{R},\le,r)$ is closed (as a subspace of $\mathcal{AR}^{\bN}$) and satisfies axioms {\bf A.1}, {\bf A.2}, {\bf A.3}, and {\bf A.4},
then every  subset of $\mathcal{R}$ with the property of Baire is Ramsey,
and every meager subset is Ramsey null;
in other words,
the triple $(\mathcal{R},\le,r)$ forms a topological Ramsey space.
\end{thm}

\begin{defn}[\cite{TodorcevicBK10}]\label{defn.5.16}
A family $\mathcal{F}\sse\mathcal{AR}$ of finite approximations is
\begin{enumerate}
\item
{\em Nash-Williams} if $a\not\sqsubseteq b$ for all $a\ne b\in\mathcal{F}$;
\item
{\em Ramsey} if for every partition $\mathcal{F}=\mathcal{F}_0\cup\mathcal{F}_1$ and every $X\in\mathcal{R}$,
there are $Y\le X$ and $i\in\{0,1\}$ such that $\mathcal{F}_i|Y=\emptyset$.
\end{enumerate}
\end{defn}

The Abstract Nash-Williams Theorem (Theorem 5.17  in \cite{TodorcevicBK10}), which follows from the Abstract Ellentuck Theorem,
will suffice for the arguments in this paper.

\begin{thm}[Abstract Nash-Williams Theorem]\label{thm.ANW}
Suppose $(\mathcal{R},\le,r)$ is a closed triple that satisfies {\bf A.1} - {\bf A.4}. Then
every Nash-Williams family of finite approximations is Ramsey.
\end{thm}

\begin{defn}\label{def.frontR1}
Suppose $(\mathcal{R},\le,r)$ is a closed triple that satisfies {\bf A.1} - {\bf A.4}.
Let $X\in\mathcal{R}$.
A family $\mathcal{F}\sse\mathcal{AR}$ is a {\em front} on $[0,X]$ if
\begin{enumerate}
\item
For each $Y\in[0,X]$, there is an $a\in \mathcal{F}$ such that $a\sqsubset Y$; and
\item
$\mathcal{F}$ is Nash-Williams.
\end{enumerate}
\end{defn}

\begin{rem}
There is also a general notion of {\em barrier} for topological Ramsey spaces (see Definition 5.18 in \cite{TodorcevicBK10}).
Everything proved for the spaces $\mathcal{E}_k$, $k\ge 2$, in this paper for fronts carries over to barriers,
 since given a  front, there is a member of the space such that, relativized to that member, the front becomes a barrier.
This follows from  Corollary 5.19 in \cite{TodorcevicBK10}, since 
 for each space $\mathcal{E}_k$, the  quasi-order $\le_{\fin}$ is actually a partial order.
Rather than defining more notions than are necessary for the main results in this paper, we provide these references for the interested reader.
\end{rem}

We finish this section by reminding the reader of  the \Pudlak-\Rodl\ Theorem for canonical equivalence relations on fronts on the Ellentuck space.

\begin{defn}\label{def.irredPR}
Let $([\om]^{\om},\sse,r)$ be the Ellentuck space.
A map $\vp$ on a front $\mathcal{F}\sse [\om]^{<\om}$ is called
\begin{enumerate}
\item
{\em inner} if for each $a\in \mathcal{F}$, 
$\vp(a)\sse a$.
\item
{\em Nash-Williams} if for all pairs $a,b\in \mathcal{F}$,
$\vp(a)\not\sqsubset \vp(b)$.
\item 
{\em irreducible} if it is inner and Nash-Williams.
\end{enumerate}
\end{defn}

\begin{thm}[\Pudlak/\Rodl,  \cite{Pudlak/Rodl82}]
Let $R$ be an equivalence relation on a front $\mathcal{F}$ on the Ellentuck space.
Then there is an irreducible map $\vp$ and an $X\in[\om]^{\om}$ such that for all $a,b\in\mathcal{F}$ with $a,b\sse X$,
\begin{equation}
a\, R\, b\ \ \llra\ \ \vp(a)=\vp(b).
\end{equation}
\end{thm}

This  theorem has been generalized to new topological Ramsey spaces in the papers \cite{Dobrinen/Todorcevic14},  \cite{Dobrinen/Todorcevic15}, and \cite{Dobrinen/Mijares/Trujillo14}.
In Section \ref{sec.rct}, we will extend it to the high dimensional Ellentuck spaces.


\section{High dimensional Ellentuck Spaces}\label{sec.defR}

We present here a new hierarchy of topological Ramsey spaces which  generalize the Ellentuck space in a natural manner.
Recall that the {\em Ellentuck space} is the triple $([\om]^{\om},\sse,r)$,
where the finitzation map $r$ is  defined as follows:  for each $X\in[\om]^{\om}$ and $n<\om$,   $r(n,X)$ is the set of   the least $n$ elements of $X$.
We shall let $\mathcal{E}_1$ denote the Ellentuck space.
It was proved by Ellentuck in \cite{Ellentuck74} that $\mathcal{E}_1$ is a topological Ramsey space.
We point out  that the members of $\mathcal{E}_1$  can be identified with the  subsets of $[\om]^{1}$ of  (lexicographical) order-type $\om$.

The first of our new spaces, $\mathcal{E}_2$,  was motivated by the problem of finding the   structure of the Tukey types of ultrafilters Tukey reducible to the generic ultrafilter forced by $\mathcal{P}(\om^2)/\Fin^{\otimes 2}$, denoted by $\mathcal{G}_2$.
In \cite{Blass/Dobrinen/Raghavan13}, it was  proved  that $\mathcal{G}_2$ is neither  maximum  nor minimum in Tukey types of nonprincipal ultrafilters.
However, this left open the question of what exactly is the structure of the Tukey types of ultrafilters Tukey reducible to $\mathcal{G}_2$.
To answer this question (which we do in Theorem \ref{thm.main}), the  first step is to construct  the  second order Ellentuck space $\mathcal{E}_2$, which comprises a dense subset of $((\Fin^{\otimes 2})^+, \sse^{\Fin^{\otimes 2}})$.
Since $\mathcal{P}(\om^2)/\Fin^{\otimes 2}$ is forcing equivalent to 
$((\Fin^{\otimes 2})^+, \sse^{\Fin^{\otimes 2}})$, 
each generic ultrafilter for $(\mathcal{E}_2,  \sse^{\Fin^{\otimes 2}})$
is generic for $\mathcal{P}(\om^2)/\Fin^{\otimes 2}$, and vice versa.
The Ramsey theory available to us through $\mathcal{E}_2$ will aid in finding  the  initial Tukey structure below $\mathcal{G}_2$.

Our  construction of $\mathcal{E}_2$ can be  generalized to find topological Ramsey spaces which are forcing equivalent to the partial orders $\mathcal{P}(\om^k)/\Fin^{\otimes k}$, for each $k\ge 2$.
Each space $\mathcal{E}_k$ is composed of members which are subsets of $[\om]^k$ which, when ordered lexicographically, are seen to have order type exactly the countable ordinal $\om^k$.
For each $k\ge 1$, the members of $\mathcal{E}_{k+1}$ look like $\om$ many copies of the members of $\mathcal{E}_k$.
These spaces will provide the structure needed to crystalize  the initial Tukey structure below the ultrafilters forced by $\mathcal{P}(\om^k)/\Fin^{\otimes k}$,  for each $k\ge 2$
(see Theorem \ref{thm.main}).

We now begin the process of defining the new class of spaces $\mathcal{E}_k$.
We start by defining a  well-ordering on non-decreasing sequences  of members of $\om$ which forms the backbone for the structure of the members in the spaces.
The explanation of why this structure was chosen, and indeed is needed, will follow Definition \ref{def.E_k}.

\begin{defn}[The well-ordered set $(\om^{\not\,\downarrow\le k},\prec)$]\label{defn.prec}
Let $k\ge 2$, and 
let $\om^{\not\,\downarrow\le k}$ denote the collection of  all non-decreasing sequences of members of $\om$ of length less than or equal to $k$.
Let $<_{\mathrm{lex}}$ denote the lexicographic ordering on  $\om^{\not\,\downarrow\le k}$, where we also consider any proper initial segment of a sequence to be lexicographically below that sequence.
Define a well-ordering $\prec$ on 
 $\om^{\not\,\downarrow\le k}$  as follows.
First,
we set the empty sequence $()$ to be the $\prec$-minimum element;
so
 for all nonempty sequences $\vec{j}$ in $\om^{\not\,\downarrow\le k}$, 
we have 
$()\prec \vec{j}$.
In general, 
given  $(j_0,\dots,j_{p-1})$  and 
$(l_0,\dots,l_{q-1})$ in  $\om^{\not\,\downarrow\le k}$ with $p,q\ge 1$,
define $ (j_0,\dots, j_{p-1})\prec(l_0,\dots,l_{q-1})$ if and only if either
\begin{enumerate}
\item
$j_{p-1}<l_{q-1}$, or
\item
$j_{p-1}=l_{q-1}$ and
$ (j_0,\dots, j_{p-1})<_{\mathrm{lex}}(l_0,\dots,l_{q-1})$.
\end{enumerate}
Since  $\prec$  well-orders  $\om^{\not\,\downarrow\le k}$ 
in order-type $\om$, 
we fix the notation of letting $\vec{j}_m$  denote the $m$-th member of $(\om^{\not\,\downarrow\le k},\prec)$.
 For $\vec{l}\in \om^{\not\,\downarrow\le k}$,
we let 
$m_{\vec{l}}\in\om$  denote  {\em the} $m$ such that  $\vec{l}=\vec{j}_{m}$.
In particular, $\vec{j}_0=()$ and $m_{()}=0$.

Let $\om^{\not\,\downarrow k}$ denote the collection of all non-decreasing sequences of length $k$ of members of $\om$.
Note that 
$\prec$ also well-orders $\om^{\not\,\downarrow k}$ in order type $\om$.
Fix the notation of letting $\vec{i}_n$ denote the $n$-th member of $(\om^{\not\,\downarrow k},\prec)$.
\end{defn}

We now define the top member $\bW_k$ of the space $\mathcal{E}_k$.
This set $\bW_k$ is the prototype for all members of $\mathcal{E}_k$ in the sense that every member of $\mathcal{E}_k$ will be a subset of $\bW_k$ which has the same structure as $\bW_k$, defined below.

\begin{defn}[The top member $\bW_k$ of $\mathcal{E}_k$]\label{defn.W_k}
Let $k\ge 2$ be given.
For each $\vec{i}=(i_0,\dots,i_{k-1})\in\om^{\not\,\downarrow  k}$, 
define 
\begin{equation}
\bW_k(\vec{i})=\{m_{\vec{i}\re p}: 1\le p\le k\}.
\end{equation}
Thus, each $\bW_k(\vec{i})$ is a member of $[\om]^k$.
Define
\begin{equation}
\bW_k=\{\bW_k(\vec{i}):\vec{i}\in \om^{\not\,\downarrow k}\}.
\end{equation}
Note that $\bW_k$ is a subset of $[\om]^k$ with order-type $\om^k$, under the lexicographical ordering.
\end{defn}

For $1\le p\le k$, letting $\bW_k(\vec{i}\re p)$ denote $\{m_{\vec{i}\re q}:1\le q\le p\}$, 
and letting $\bW_k( ())=\emptyset$,
we see that 
 $\bW_k$  induces the tree 
$\widehat{\bW}_k=\{\bW_k(\vec{j}):\vec{j}\in\om^{\not\,\downarrow \le k}\}\sse[\om]^{\le k}$
obtained by taking all initial segments of members of $\bW_k$.
The key points about the structure of 
 $\widehat{\bW}_k$ are the following, which will be essential in the next definition:
\begin{enumerate}
\item[(ii)]
For each $m\ge 1$,
$\max(\bW_k(\vec{j}_m))<\max(\bW_k(\vec{j}_{m+1}))$.
\item[(iii)]
For all $\vec{j},\vec{l}\in\om^{\not\,\downarrow \le k}$,
$\bW_k(\vec{j})$ is an initial segment of $\bW_k(\vec{l})$ if and only if $\vec{j}$ is an initial segment of $\vec{l}$.
\end{enumerate}
All members of the space $\mathcal{E}_k$ will have this structure.

\begin{defn}[The spaces $(\mathcal{E}_k,\le,r)$, $k\ge 2$]\label{def.E_k}
For $\vec{j}_m\in\om^{\not\,\downarrow\le k}$, let $|\vec{j}|$ denote
the length of the sequence $\vec{j}$.
We say that $\widehat{X}$ is an {\em $\mathcal{E}_k$-tree} if
$\widehat{X}$ is a  function from $\om^{\not\,\downarrow  \le k}$ into $\widehat{\bW}_k$ such that
\begin{enumerate}
\item[(i)]
For each $m<\om$,
$\widehat{X}(\vec{j}_m)\in[\om]^{|\vec{j}_m|}\cap \widehat{\bW}_k$;
\item[(ii)]
For all $1\le m<\om$,
$\max(\widehat{X}(\vec{j}_m))<\max(\widehat{X}(\vec{j}_{m+1}))$;
\item[(iii)]
For all $m,n<\om$,
$\widehat{X}(\vec{j}_m)\sqsubset \widehat{X}(\vec{j}_n)$ if and only if 
$\vec{j}_m\sqsubset \vec{j}_n$.
\end{enumerate}

For $\widehat{X}$ an $\mathcal{E}_k$-tree, let $[\widehat{X}]$ denote
the function $\widehat{X}\cap (\om^{\not\, \downarrow k}\times\bW_k)$.
Define the space 
$\mathcal{E}_k$ to be the collection of all $[\widehat{X}]$ such that  $\widehat{X}$ is an
$\mathcal{E}_k$-tree.
Thus, $\mathcal{E}_k$ is the space of all functions $X$ from $\om^{\not\, \downarrow k}$ into $\bW_k$ which induce an $\mathcal{E}_k$-tree.

For $X,Y\in\mathcal{E}_k$,  define $Y\le X$ if and only if $\ran(Y)\sse \ran(X)$.
For each $n<\om$, the $n$-th finite aproximation $r_n(X)$  
is $X\cap(\{\vec{i}_p:p<n\}\times\bW_k)$.
As usual, we let
$\mathcal{AR}$ denote the collection $\{r_n(X):X\in\mathcal{E}_k$ and $n<\om\}$.
For $a,b\in\mathcal{AR}$ 
define $a\le_{\mathrm{fin}} b$ if and only if $\ran(a)\sse\ran(b)$.
\end{defn}

\begin{rem}
The  members of $\mathcal{E}_k$ are functions from 
$\om^{\not\, \downarrow k}$ into $\bW_k$ which are obtained by restricting $\mathcal{E}_k$-trees to their maximal nodes.
Each member of $\mathcal{E}_k$ uniquely determines an $\mathcal{E}_k$-tree and vice versa.
We will  identify each member $X$ of $\mathcal{E}_k$ with  its image $\ran(X)=\{X(\vec{i}_n):n<\om\}\sse \bW_k$, as this identification is unambiguous.
In this vein, we may think of $r_n(X)$ as $\{X(\vec{i}_p):p<n\}$.
\end{rem}

Define the projection maps $\pi_l$, $l\le k$,   as follows.
For all  $\vec{j}\in\om^{\not\,\downarrow \le k}$,
define $\pi_0(\{m_{\vec{j}\re q}:1\le q\le |\vec{j}|\})=\emptyset$.
For $1\le l\le k$,
for all   $\vec{j}\in\om^{\not\,\downarrow \le k}$ with $|\vec{j}|\ge l$,
define 
\begin{equation}
\pi_l(\{m_{\vec{j}\re q}:1\le q\le |\vec{j}|\})=\{m_{\vec{j}\re q}:1\le q\le l\}.
\end{equation}
Thus, $\pi_l$ is defined on those members of $\widehat{\bW}_k$ with length at least $l$, and projects to their initial segments of length $l$.

For each $l\le k$, we let $N^k_l$ denote the collection of all $n\in\om$ such that 
given $a\in\mathcal{AR}_n$,
for all $b\in r_{n+1}[a,\bW_k]$, $\pi_l(b(\vec{i}_n))\in\pi_l(a)$,
 but $\pi_{l+1}(b(\vec{i}_n))\not\in\pi_{l+1}(a)$.

\begin{rem}
In defining the spaces $\mathcal{E}_k$, there is a tension between needing 
 the members of $\mathcal{E}_k$ to have  order-type $\om^k$ 
and needing  the finitization map $r$ to  give back any member of $\mathcal{E}_k$ in $\om$ many steps.
Thus, it was necessary to find a way to diagonalize through a set $X\sse[\om]^k$ of order-type $\om^k$ in $\om$ many steps in such a way that the axioms \bf A.1 \rm - \bf A.4 \rm hold.
All the Axioms except for \bf A.3 \rm (b)  could be proved using several different choices for the finitzation map $r$.
However,  the structure of the well-ordering $(\om^{\not\,\downarrow k},\prec)$, the structure of $\bW_k$ given as a template for the members of $\mathcal{E}_k$, and 
conditions (2) and (3) in Definition
\ref{def.E_k}
are precisely what allow us to prove axiom \bf A.3 \rm (b), which will be  proved in Lemma \ref{lem.A.3}.
Interestingly, the Pigeonhole Principle \bf A.4 \rm is actually more straightforward than \bf A.3 \rm to prove for these spaces.
\end{rem}

Before proving that these $\mathcal{E}_k$ form topological Ramsey spaces, we begin with 
 some concrete examples starting with $\mathcal{E}_2$.

\begin{example}[The space $\mathcal{E}_2$]\label{ex.E_2}
The members of $\mathcal{E}_2$ look like $\om$ many copies of the Ellentuck space; that is, each member has order-type $\om\cot\om$, under the lexicographic order.
The well-order $(\om^{\not\,\downarrow \le 2}, \prec)$  begins as follows:
\begin{equation}
()\prec (0)\prec (0,0)\prec (0,1)\prec (1)\prec (1,1)\prec (0,2)\prec (1,2)\prec (2)\prec (2,2)\prec\cdots
\end{equation}

The tree structure of $\om^{\not\,\downarrow \le 2}$, under lexicographic order, looks $\om$ copies of $\om$, and has order type the countable ordinal  $\om^2$ under the  lexicographic ordering.
Here, we picture the finite tree $\{ \vec{j}_m: m<22\}$, which indicates how the rest of the tree $\om^{\not\,\downarrow \le 2}$ is formed.
This is the same as  the tree formed by taking all initial segments of the set $\{\vec{i}_n:n<15\}$.

\begin{figure}[\h]
\centering
{\footnotesize
\begin{tikzpicture}[scale=.7,grow'=up, level distance=40pt,sibling distance=.2cm]
\tikzset{grow'=up}
\Tree [.$()$ [.$(0)$ [.$\rotatebox{45}{(0,0)}$ ][.$\rotatebox{45}{(0,1)}$ ][.$\rotatebox{45}{(0,2)}$ ][.$\rotatebox{45}{(0,3)}$ ]  [.$\rotatebox{45}{(0,4)}$ ]   ] [.$(1)$ [.$\rotatebox{45}{(1,1)}$ ][.$\rotatebox{45}{(1,2)}$ ][.$\rotatebox{45}{(1,3)}$ ][.$\rotatebox{45}{(1,4)}$ ] ] [.$(2)$ [.$\rotatebox{45}{(2,2)}$ ] [.$\rotatebox{45}{(2,3)}$ ][.$\rotatebox{45}{(2,4)}$ ] ] [.$(3)$ [.$\rotatebox{45}{(3,3)}$ ] [.$\rotatebox{45}{(3,4)}$ ] ]  [.$(4)$  [.$\rotatebox{45}{(4,4)}$ ] ] ]
\end{tikzpicture}
}
\caption{$\omega^{\not\downarrow \le 2}$}
\end{figure}

The $\prec$ ordering on $\om^{\not\,\downarrow \le 2}$ determines the nodes in $\widehat{\bW}_2$.
Technically, the maximal nodes in the figure below show $r_{15}(\bW_2)$, which indicates how the rest of $\bW_2$ is formed.

\begin{figure}[\h]
\centering
{\footnotesize
\begin{tikzpicture}[scale=.7,grow'=up, level distance=40pt,sibling distance=.08cm]
\tikzset{grow'=up}
\Tree [.$\emptyset$ [.$\{0\}$ [.$\rotatebox{45}{\{0,1\}}$ ][.$\rotatebox{45}{\{0,2\}}$ ][.$\rotatebox{45}{\{0,5\}}$ ][.$\rotatebox{45}{\{0,9\}}$ ] [.$\rotatebox{45}{\{0,14\}}$ ]] [.$\{3\}$ [.$\rotatebox{45}{\{3,4\}}$ ][.$\rotatebox{45}{\{3,6\}}$ ][.$\rotatebox{45}{\{3,10\}}$ ] [.$\rotatebox{45}{\{3,15\}}$ ]] [.$\{7\}$ [.$\rotatebox{45}{\{7,8\}}$ ] [.$\rotatebox{45}{\{7,11\}}$ ]  [.$\rotatebox{45}{\{7,16\}}$ ]] [.$\{12\}$ [.$\rotatebox{45}{\{12,13\}}$ ] [.$\rotatebox{45}{\{12,17\}}$ ]  ] [.$\{18\}$ [.$\rotatebox{45}{\{18,19\}}$ ]  ] ]
\end{tikzpicture}
}
\caption{$\mathbb{W}_{2}$}
\end{figure}
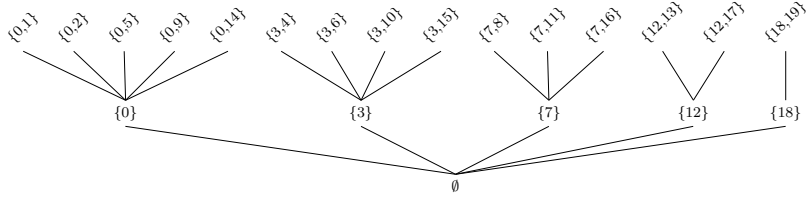

We now present some typical finite approximations to members of $\mathcal{E}_2$.

\begin{figure}[\h]
\centering
{\footnotesize
\begin{tikzpicture}[scale=.7,grow'=up, level distance=40pt,sibling distance=.4cm]
\tikzset{grow'=up}
\Tree [.$\emptyset$ [.$\{3\}$ [.$\{3,6\}$ ][.$\{3,15\}$ ] [.$\{3,36\}$ ] ]
 [.$\{25\}$ [.$\{25,32\}$ ] [.$\{25,40\}$ ]  ] [.$\{42\}$ [.$\{42,51\}$ ] ] ]
\end{tikzpicture}
}
\caption{$r_6(X)$ for a typical $X\in\mathcal{E}_2$}
\end{figure}
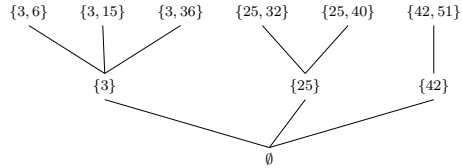

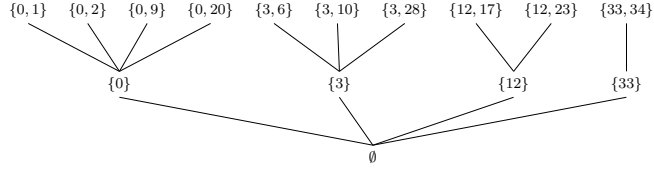
\begin{figure}[\h]
\centering
{\footnotesize
\begin{tikzpicture}[scale=.7,grow'=up, level distance=40pt,sibling distance=.2cm]
\tikzset{grow'=up}
\Tree [.$\emptyset$ [.$\{0\}$ [.$\{0,1\}$ ][.$\{0,2\}$ ][.$\{0,9\}$ ]  [.$\{0,20\}$ ]  ] [.$\{3\}$ [.$\{3,6\}$ ][.$\{3,10\}$ ][.$\{3,28\}$ ]  ]  [.$\{12\}$ [.$\{12,17\}$ ]  [.$\{12,23\}$ ] ]  [.$\{33\}$ [.$\{33,34\}$ ]] ]
\end{tikzpicture}
}
\caption{$r_{10}(X)$ for a typical $X\in\mathcal{E}_2$}
\end{figure}
\end{example}

The following trivial fact is stated, as it is important to seeing that the space in this paper is forcing equivalent to the forcing considered in \cite{Blass/Dobrinen/Raghavan13}.

\begin{fact}\label{fact.dense}
For any set $S\sse [\om]^2$ such that for infinitely many $i\in\pi_1(S)$,
the set $\{j\in\om:\{i,j\}\in S\}$ is infinite,
 there is an $X\in\mathcal{E}_2$ such that $X\sse S$.
\end{fact}

Let $\om^2$ denote $\om\times\om$ and let $\Fin^{\otimes 2}$ denote  the ideal $\Fin\times\Fin$,
which is
 the collection of all subsets $A$ of $\om\times\om$ such  that for all but finitely many $i\in \om$, the fiber $A(i):=\{j<\om:(i,j)\in A\}$ is finite.
Abusing notation, we also let $\Fin^{\otimes 2}$ denote the ideal on $[\om]^2$ consisting of sets $A\sse[\om]^2$ such that for all but finitely many $i\in\om$, the set $\{j>i:\{i,j\}\in A\}$ is finite.
Given $X,Y\sse [\om]^2$, we write $Y\sse^{\Fin^{\otimes 2}} X$ if and only if  $Y\setminus X\in\Fin^{\otimes 2}$.
We now point out how our space $\mathcal{E}_2$  partially ordered by $\sse^{\mathrm{Fin}^{\otimes 2}}$ is forcing equivalent to 
 $\mathcal{P}(\om^2)/\mathrm{Fin}^{\otimes 2}$.

\begin{prop}\label{prop.forcingequiv}
 $(\mathcal{E}_2,\sse^{\mathrm{Fin}^{\otimes 2}})$ is forcing equivalent to  $\mathcal{P}(\om^2)/\mathrm{Fin}^{\otimes 2}$.
\end{prop}

\begin{proof}
It is well-known that  $\mathcal{P}(\om^2)/\mathrm{Fin}^{\otimes 2}$ is forcing equivalent to 
$((\Fin\times\Fin)^+, \sse^{\mathrm{Fin}^{\otimes 2}})$,
where $(\Fin\times\Fin)^+$ is the collection of all subsets $A\sse \om^2$ such that for infinitely many coordinates $i$, the $i$-th fiber of $A$ is infinite.
(See, for instance, \cite{Blass/Dobrinen/Raghavan13}.)
Identifying 
 $\{(i,j):i<j<\om\}$ with $[\om]^2$, we see that 
 the collection of all infinite subsets of $[\om]^2$ with lexicographic order-type exactly $\om^2$ forms a $\sse$-dense subset of $(\Fin\times\Fin)^+$.
Further, for each $Z\sse[\om]^2$ with lexicographic order-type exactly $\om^2$, there is an $X\in\mathcal{E}_2$ such that
$X\sse Z$.
Thus, $(\mathcal{E}_2,\sse^{\Fin^{\otimes 2}})$
is forcing equivalent to 
$\mathcal{P}(\om\times\om)/\mathrm{Fin}^{\otimes 2}$.
\end{proof}

Next we present the specifics of the structure of the space $\mathcal{E}_3$.

\begin{example}[The space $\mathcal{E}_3$]\label{ex.E_3}

The well-order $(\om^{\not\,\downarrow\le 3},\prec)$ begins as follows: 

\begin{align}
\emptyset &\prec (0) \prec (0,0)\prec (0,0,0)\prec (0,0,1)\prec (0,1)\prec (0,1,1)\prec (1)\cr
&\prec (1,1)
\prec (1,1,1)\prec
(0,0,2)\prec (0,1,2)\prec (0,2)\prec 
(0,2,2)
\cr
&\prec (1,1,2)\prec (1,2)\prec (1,2,2)\prec (2)\prec (2,2)\prec (2,2,2)\prec (0,0,3)\prec \cdots
\end{align}

The set $\om^{\not\,\downarrow \le 3}$ is a tree of height three with each non-maximal node branching into $\om$ many nodes.  
The maximal nodes in the following figure  is technically the set  $\{\vec{i}_m:m<20\}$,  which indicates  the structure of $\om^{\not\,\downarrow \le 3}$.

\begin{figure}[\h]
\centering
{\footnotesize
\begin{tikzpicture}[scale=.6,grow'=up, level distance=30pt,sibling distance=.1cm]
\tikzset{grow'=up}
\Tree [.$\emptyset$ [.$(0)$ [.$(0,0)$ [.$\rotatebox{45}{(0,0,0)}$ ][.$\rotatebox{45}{(0,0,1)}$ ][.$\rotatebox{45}{(0,0,2)}$ ] [.$\rotatebox{45}{(0,0,3)}$ ] ][.$(0,1)$ [.$\rotatebox{45}{(0,1,1)}$ ][.$\rotatebox{45}{(0,1,2)}$ ] [.$\rotatebox{45}{(0,1,3)}$ ]][.$(0,2)$ [.$\rotatebox{45}{(0,2,2)}$ ]  [.$\rotatebox{45}{(0,2,3)}$ ]]  [.$(0,3)$ [.$\rotatebox{45}{(0,3,3)}$ ] ]   ][.$(1)$ [.$(1,1)$ [.$\rotatebox{45}{(1,1,1)}$ ][.$\rotatebox{45}{(1,1,2)}$ ] [.$\rotatebox{45}{(1,1,3)}$ ] ] [.$(1,2)$ [.$\rotatebox{45}{(1,2,2)}$ ]  [.$\rotatebox{45}{(1,2,3)}$ ]] [.$(1,3)$ [.$\rotatebox{45}{(1,3,3)}$ ] ] ][.$(2)$ [.$(2,2)$ [.$\rotatebox{45}{(2,2,2)}$ ] [.$\rotatebox{45}{(2,2,3)}$ ] ] [.$(2,3)$ [.$\rotatebox{45}{(2,3,3)}$ ] ] ][.$(3)$  [.$(3,3)$ [.$\rotatebox{45}{(3,3,3)}$ ] ] ]]
\end{tikzpicture}
}
\caption{$\omega^{\not\downarrow \le 3}$}
\end{figure}

Technically, the following figure presents $r_{20}(\bW_3)$, though the intent is to give the reader an idea of the structure of $\bW_3$.

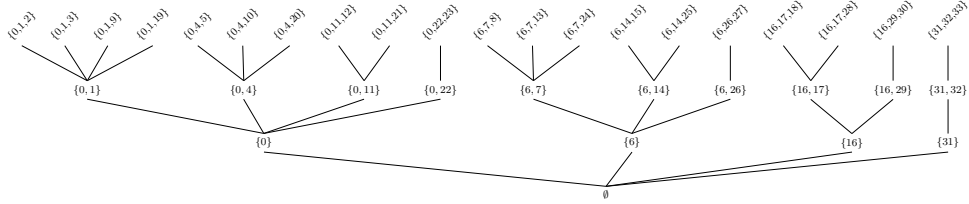
\begin{figure}[\h]
\centering
{\footnotesize
\begin{tikzpicture}[scale=.5,grow'=up, level distance=40pt,sibling distance=.1cm]
\tikzset{grow'=up}
\Tree [.$\emptyset$ [.$\{0\}$ [.$\{0,1\}$ [.$\rotatebox{45}{\{0,1,2\}}$ ][.$\rotatebox{45}{\{0,1,3\}}$ ][.$\rotatebox{45}{\{0,1,9\}}$ ][.$\rotatebox{45}{\{0,1,19\}}$ ]
 ][.$\{0,4\}$ [.$\rotatebox{45}{\{0,4,5\}}$ ][.$\rotatebox{45}{\{0,4,10\}}$ ] [.$\rotatebox{45}{\{0,4,20\}}$ ]][.$\{0,11\}$ [.$\rotatebox{45}{\{0,11,12\}}$ ]  [.$\rotatebox{45}{\{0,11,21\}}$ ]]  [.$\{0,22\}$ [.$\rotatebox{45}{\{0,22,23\}}$ ] ]  ][.$\{6\}$ [.$\{6,7\}$ [.$\rotatebox{45}{\{6,7,8\}}$ ][.$\rotatebox{45}{\{6,7,13\}}$ ][.$\rotatebox{45}{\{6,7,24\}}$ ] ] [.$\{6,14\}$ [.$\rotatebox{45}{\{6,14,15\}}$ ] [.$\rotatebox{45}{\{6,14,25\}}$ ]]  [.$\{6,26\}$ [.$\rotatebox{45}{\{6,26,27\}}$ ] ]][.$\{16\}$ [.$\{16,17\}$ [.$\rotatebox{45}{\{16,17,18\}}$ ]  [.$\rotatebox{45}{\{16,17,28\}}$ ] ] [.$\{16,29\}$ [.$\rotatebox{45}{\{16,29,30\}}$ ]  ] ][.$\{31\}$ [.$\{31,32\}$ [.$\rotatebox{45}{\{31,32,33\}}$ ] ] ] ]
\end{tikzpicture}
}
\caption{$\mathbb{W}_{3}$}
\end{figure}

We next present typical fourth and fifth approximations.

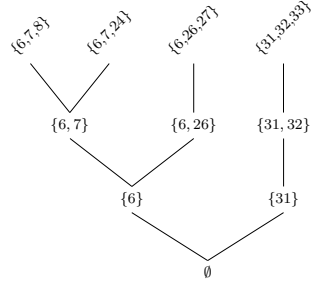
\begin{figure}[\h]
\centering
{\footnotesize
\begin{tikzpicture}[scale=.7,grow'=up, level distance=40pt,sibling distance=.4cm]
\tikzset{grow'=up}
\Tree [.$\emptyset$ [.$\{6\}$ [.$\{6,7\}$ [.$\rotatebox{45}{\{6,7,8\}}$ ][.$\rotatebox{45}{\{6,7,24\}}$ ] ]   [.$\{6,26\}$ [.$\rotatebox{45}{\{6,26,27\}}$ ] ]][.$\{31\}$ [.$\{31,32\}$ [.$\rotatebox{45}{\{31,32,33\}}$ ] ] ] ]
\end{tikzpicture}
}
\caption{$r_4(X)$ for a typical $X\in\mathcal{E}_3$}
\end{figure}

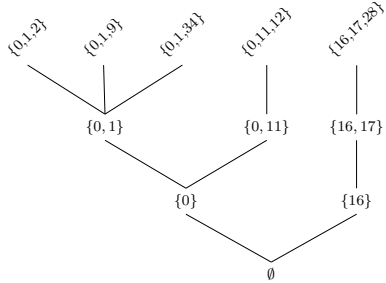
\begin{figure}[\h]
\centering
{\footnotesize
\begin{tikzpicture}[scale=.7,grow'=up, level distance=40pt,sibling distance=.4cm]
\tikzset{grow'=up}
\Tree [.$\emptyset$ [.$\{0\}$ [.$\{0,1\}$ [.$\rotatebox{45}{\{0,1,2\}}$ ][.$\rotatebox{45}{\{0,1,9\}}$ ][.$\rotatebox{45}{\{0,1,34\}}$ ]
 ][.$\{0,11\}$ [.$\rotatebox{45}{\{0,11,12\}}$ ] ]    ][.$\{16\}$ [.$\{16,17\}$ [.$\rotatebox{45}{\{16,17,28\}}$ ] ] ]]
\end{tikzpicture}
}
\caption{$r_5(X)$ for a typical $X\in\mathcal{E}_3$}
\end{figure}

By $\Fin^{\otimes 3}$, we denote $\Fin\otimes\Fin^{\otimes 2}$, which consists of all subsets  $F\sse \om^3$ such that for all  but finitely many $i\in\om$, $\{(j,k):(i,j,k)\in F\}$ is in $\Fin^{\otimes 2}$.
Identifying $[\om]^3$ with $\{(i,j,k)\in\om^3:i<j<k\}$,
we abuse notation and let $\Fin^{\otimes 3}$ on $[\om]^3$ denote the collection of all subsets $F\sse[\om]^3$ such that $\{(i,j,k):\{i,j,k\}\in F\}$ is in $\Fin^{\otimes 3}$ as defined on $\om^3$. 
It is routine to check that $(\mathcal{E}_3,\sse^{\Fin^{\otimes 3}})$ is 
forcing equivalent to $\mathcal{P}(\om^3)/\Fin^{\otimes 3}$.
\end{example}

We shall now show that for each $k\ge 2$, the space $(\mathcal{E}_k,\le,r)$  is a topological Ramsey space; hence, every  subset of $\mathcal{E}_k$  with the property of Baire is Ramsey.
Since   $\mathcal{E}_k$  is a closed subspace of $\mathcal{AR}^{\om}$,
it suffices, by  the Abstract Ellentuck Theorem  (Theorem \ref{thm.AET}), 
to show that  $(\mathcal{E}_k,\le,r)$
 satisfies the axioms \bf A.1 \rm -  \bf A.4\rm.
As it is routine to check that $(\mathcal{E}_k,\le,r)$  satisfies the axioms  \bf A.1 \rm  and  \bf A.2\rm, we leave this to the reader.
We will  show that \bf A.3 \rm holds for $\mathcal{E}_k$ for all $k\ge 2$.
Then we will show by induction on $k\ge 2$ that \bf A.4 \rm holds for $\mathcal{E}_k$.

For each fixed $k\ge 2$,
recall our  convention that $\lgl \vec{i}_n:n<\om\rgl$ is the  $\prec$-increasing enumeration of the well-ordered set $(\om^{\not\,\downarrow k},\prec)$.
Though technically each $a\in\mathcal{AR}$ 
 is a subset of $[\om]^k$, 
we shall abuse notation and use  $\max a$ to denote $\max \bigcup a$.
Recall that for $a\in\mathcal{AR}$ and $X\in\mathcal{E}_k$, 
$\depth_X(a)$ is defined to be the smallest $n$ for which $a\sse r_n(X)$, if $a\sse X$, and $\infty$ otherwise.
As is convention,  $[n,X]$ is used to denote $[r_n(X),X]$.

\begin{lem}\label{lem.A.3}
For each $ k\ge 2$, the space $(\mathcal{E}_k,\le,r)$ satisfies Axiom \bf A.3\rm.
\end{lem}

\begin{proof}
To see that \bf A.3 \rm (a) holds,
suppose that $\depth_B(a)=d<\infty$
and $A\in[d,B]$.
Let  $b= r_{d}(B)$.
Then
 $a\sse b$, $\max a=\max b$, and $[b,B]=[d,B]$.
We will recursively build a $C\in[a, A]$, which will show that $[a,A]$ is non-empty.
Let $m=|a|$.
Note that $a\in\mathcal{AR}_m|A$, since $a\sse b$ and $A\in[b,B]$.
Let $c_m$ denote $a$.

Suppose $n\ge m$ and we have already chosen  $c_n\in r_n[a, A]$ such that
 $c_n\sqsupset c_{n-1}$ if $n>m$.
Construct $c_{n+1}\in r_{n+1}[c_n,A]$ as follows.
Let $l$ be the integer less than $k$  such that $n\in N^k_l$.
Recall that 
$n\in N^k_l$
 means that every extension $c'\sqsupset c_n$ has $\pi_l(c'(\vec{i}_n))\in\pi_l(c_n)$, and $\pi_{l+1}(c'(\vec{i}_n))\not\in\pi_{l+1}(c_n)$.
If $l=0$, then  choose  $c_{n+1}(\vec{i}_n)$ to be any member of  $A$ such that $\pi_1(c_{n+1}(\vec{i}_n))>\max c_n$.
Now suppose that $l\ge 1$.
Then letting $p$ be any integer less than $n$ such that  $\vec{i}_p\re l=\vec{i}_n\re l$, we note that 
$\pi_l(c_{n+1}(\vec{i}_n))$ is predetermined to be equal to  the  set 
$\pi_l(c_n(\vec{i}_p))$.
Choose  $c_{n+1}(\vec{i}_n)$ to be any member of  $A$ such that
$\pi_l(c_{n+1}(\vec{i}_n))=\pi_l(c_n(\vec{i}_p))$
and 
$\max c_{n+1}(\vec{i}_n\re (l+1))>\max c_n$.
Define $c_{n+1}$ to be $c_n\cup c_{n+1}(\vec{i}_n)$.

In this manner, we construct a sequence $c_n$, $n\ge m$, such that each $c_{n+1}\in r_{n+1}[c_n,A]$.
Letting $C=\bigcup_{n\ge m} c_n$, we see that $C$ is in $[a,A]$;
hence  \bf A.3 \rm (a) holds.

To see that \bf A.3 \rm(b) holds, 
suppose $A\le B$ and $[a,A]\ne\emptyset$.
We will construct an $A'\in[\depth_B(a),B]$ such that 
$\emptyset\ne[a,A']\sse[a,A]$.
Let $d=\depth_B(a)$ and  let  $a'_d=r_d(B)$.
For each $n\ge d$,  given $a'_n$, 
 we  will  choose $a'_{n+1}\in\mathcal{AR}_{n+1}$
such that 
\begin{enumerate}
\item
 $a'_{n+1}\in r_{n+1}[a'_{n}, B]$;
and 
\item
If $n\in N^k_l$ and $a'_n(\vec{i}_n\re l)\in\pi_l(A)$,
then 
$a'_{n+1}(\vec{i}_n)\in A$.
\end{enumerate}

Let $n\ge d$, and suppose   $a'_n$ has been chosen satisfying (1) and (2).
Choose $a'_{n+1}(\vec{i}_n)$ as follows.
Let $l< k$ be the integer such that  $n\in N^k_l$.
If $l=0$, then choose $a'_{n+1}(\vec{i}_n)$ to be any member of $A$ such that $\max\pi_1(a'_{n+1}(\vec{i}_n))>\max a'_n$.

Suppose now that $l\ge 1$.
We have two cases.

\it Case 1. \rm  $\vec{i}_n\re l=\vec{i}_m\re l$ for some $m<d$.
Then
$a_{n+1}'(\vec{i}_n)$
 must be  chosen so that  $\pi_l(a'_{n+1}(\vec{i}_n))\in\pi_l(b)$.
In the case that 
$a'_d(\vec{i}_m\re l)$ is in $\pi_l(a)$,
then  we can
 choose $a'_{n+1}(\vec{i}_n)\in A$ such that 
$a'_{n+1}(\vec{i}_n\re l)=a'_d(\vec{i}_m\re l)$
and $\max \pi_{l+1}(a'_{n+1}(\vec{i}_n))>\max  a'_n$.
In the case that $a'_d(\vec{i}_m\re l)$ is in $\pi_l(b)\setminus \pi_l(a)$,
 there is no way to choose $a'_{n+1}(\vec{i}_n)$ to be a member of $A$;
 so we choose $a'_{n+1}(\vec{i}_n)$ to be a member of $B$ such that
$a'_{n+1}(\vec{i}_n\re l)=a'_d(\vec{i}_m\re l)$ and
 $\max \pi_{l+1}(a'_{n+1}(\vec{i}_n))>\max a'_n$.

\it Case 2. \rm  $\vec{i}_n\re l\ne\vec{i}_m\re l$ for any $m<d$.
In this case, 
$\pi_l(a'(\vec{i}_n))$ cannot be in $\pi_l(b)$.
Since $l\ge 1$, there must be some $d\le m<n$ such that $\vec{i}_n\re l=\vec{i}_m\re l$.
By our construction,
$a'_n(\vec{i}_m\re l)$  must be in $\pi_l(A)$.
Choose $a'_{n+1}(\vec{i}_n)$ to be any member of $A$ 
such that $\pi_l(a'_{n+1}(\vec{i}_n))=
 \pi _l(a'_n(\vec{i}_m))$ with $\max\pi_{l+1}(a'_{n+1}(\vec{i}_n))>\max a_n$.

Having chosen $a'_{n+1}(\vec{i}_n)$,
let $a'_{n+1}=a'_n\cup \{a'_{n+1}(\vec{i}_n)\}$.
In this manner, we form a sequence $\lgl a'_n:n\ge d\rgl$ satisfying (1) and (2).
Let $A'=\bigcup_{n\ge d} a'_n$.
By construction, $A'$ is a member of $\mathcal{E}_k$ and $A'\in [b,B]$.
Since $a\sse A'$, $\emptyset\ne[a,A']$.

To see that 
$[a,A']\sse[a,A]$, 
let $X$ be any member of $[a,A']$.
For each $m<\om$, let  $n_m$ be such that $X(\vec{i}_m)=A'(\vec{i}_{n_m})$.
We show that $X(\vec{i}_m)\in A$.
Let $l<k$ be such that $m\in N^k_l$.
If 
 $\pi_l(X(\vec{i}_m))\in\pi_l(a)$,
then $A'(\vec{i}_{n_m})$ was chosen to be in $A$; 
thus $X(\vec{i}_m)\in A$.
Otherwise, 
$\pi_l(X(\vec{i}_m))\not\in\pi_l(a)$.
Since $X\sqsupset a$, 
it must be the case that $\min X(\vec{i}_m)>\max a$.
 Thus,
 $A'(\vec{i}_{n_m})$ was chosen to be in $A$;
hence, $X(\vec{i}_m)\in A$.
Therefore, $X\sse  A$, so $X\in [a,A]$.
\end{proof}

\begin{rem}
Our choice of finitization using the structure of the well-ordering $(\om^{\not\,\downarrow \le k},\prec)$ 
was made precisely so that \bf A.3 \rm (b)  could be proved.
In earlier versions of this work, we used larger finitzations so that each member $a\in\mathcal{AR}_m$ would contain precisely $m$ members $a(\vec{i}_n)$ with $n\in N^k_0$.  
This had the advantage that
the ultrafilters constructed  using fronts $\mathcal{AR}_m$ as base sets would be naturally seen as Fubini products of $m$ many ultrafilters.
However,  \bf A.3 \rm   (b) did not hold under that approach, and as such, we had to prove the Abstract Nash-William Theorem  directly from the other three and a half axioms.
Our former approach still provided the initial Tukey structures,  but our finitization in this paper map makes it clear that these new spaces really are generalizations of the Ellentuck space and saves us from some unnecessary redundancy.
Moreover, the approach we use has the advantage of allowing for new generalizations of the \Pudlak-\Rodl\ Theorem to the spaces $\mathcal{E}_k$.
\end{rem}

Towards proving \bf A.4 \rm for $\mathcal{E}_2$, 
we first prove a lemma showing that there are three canonical equivalence relations for 1-extensions on the space $\mathcal{E}_2$.
This fact is already known for the partial ordering $((\Fin^{\otimes 2})^+,\sse)$ (see Corollary 33 in \cite{Blass/Dobrinen/Raghavan13}); we are merely  making it precise in the context of our space $\mathcal{E}_2$.
Given $s\in\mathcal{AR}$, 
we shall say that $t$ is a {\em 1-extension} of $s$ if
 $t\in r_{|s|+1}[s, \bW_2]$.
For $n\in N^2_0$, $s\in\mathcal{AR}_{n}$, and $Y\contains s$, we shall say that a function $f:r_{n+1}[s,Y]\ra\om$ is 
{\em constant on blocks} if
for all  $t,u\in r_{n+1}[s,Y]$,
$f(t)=f(u)$ $\llra$  $\pi_1(t(\vec{i}_n))=\pi_1(u(\vec{i}_n))$.

\begin{lem}[Canonical Equivalence Relations on 1-Extensions in $\mathcal{E}_2$]\label{lem.canon1ext}
Suppose  $n<\om$, $s\in\mathcal{AR}_{n}$
and $s\sse X$, and let
$f:r_{n+1}[s,X]\ra\om$.
Then there is a $Y\in [\depth_X(s),X]$ such that 
 $f\re r_{n+1}[s,Y]$ satisfies exactly one of the following:
\begin{enumerate}
\item
$f\re r_{n+1}[s,Y]$  is one-to-one;
\item
$f\re r_{n+1}[s,Y]$  is constant on blocks;
\item
$f\re r_{n+1}[s,Y]$  is constant.
\end{enumerate}
Moreover,  (2) is impossible if $n\in N^2_1$.
\end{lem}

\begin{proof}
\it Case 1. \rm
$n\in N^2_1$.
Let $\vec{j}_m$ be the member of $\om^{\not\,\downarrow 1}$ such that $\vec{j}_m=\vec{i}_n\re 1$.
Suppose there is an infinite subset $P\sse \om\setminus n$ such that
for each $p\in P$, $\vec{i}_p\re 1=\vec{j}_m$ 
and $f$ is one-to-one on 
$\{s\cup X(\vec{i}_p):p\in P\}$.
Then by Fact \ref{fact.dense},
 there is a $Y\in [\depth_X(s),X]$ such that each $t\in r_{n+1}[s,Y]$ has $t(\vec{i}_n) = X(\vec{i}_p)$ for some $p\in P$.
It follows that $f$ is one-to-one on $r_{n+1}[s,Y]$.
Otherwise, there is an infinite subset $P\sse \om\setminus n$ such that
for each $p\in P$, $\vec{i}_p\re 1=\vec{j}_m$, 
and
 $f$ is constant on 
$\{s\cup X(\vec{i}_p): p\in P\}$.
By Fact \ref{fact.dense}, there is a $Y\in [\depth_X(s),X]$ such that each $t\in r_{n+1}[s,Y]$ has $t(\vec{i}_n) = X(\vec{i}_p)$ for some $p\in P$.
Then $f$ is constant on $r_{n+1}[s,Y]$.

\it Case 2. \rm $n\in N^2_0$.
Suppose  there are infinitely many  $m$ for which  $f$ is one-to-one on the set $\{s\cup X(\vec{i}_p):\vec{i}_p\re 1=\vec{j}_m\}$.
Then there is a $Y\in  [\depth_X(s),X]$ such that $f$ is one-to-one on $r_{n+1}[s,Y]$.

Suppose now that  there are infinitely many 
$m<\om$ for which there is an infinite set $P_m\sse\{p\in\om:\vec{i}_p\re 1=\vec{j}_m\}$ such that $f$ is constant  on the set $\{s\cup X(\vec{i}_p):p\in P_m\}$.
If the value of $f$ on $\{s\cup X(\vec{i}_p):p\in P_m\}$ is different for infinitely many $m$,
then, applying  Fact \ref{fact.dense}, there is a $Y\in [\depth_X(s),X]$ such that $f$ is constant on blocks  on $r_{n+1}[s,Y]$.
If the value of $f$ on $\{s\cup X(\vec{i}_p):p\in P_m\}$ is the same for infinitely many $m$,
then, applying  Fact \ref{fact.dense}, there is a $Y\in [\depth_X(s),X]$ such that $f$ is constant on $r_{n+1}[s,Y]$.
\end{proof}

\begin{lem}[\bf A.4 \rm for $\mathcal{E}_2$]\label{lem.A.4}
Let $a\in\mathcal{AR}_{n}$, 
$X\in\mathcal{E}_2$ such that $X\contains a$,
 and $\mathcal{H}\sse \mathcal{AR}_{n+1}$ be given.
Then there is a $Y\in [\depth_X(a),X]$ such that either 
$ r_{n+1}[a,Y]\sse\mathcal{H}$ or else $r_{n+1}[a,Y]\cap\mathcal{H}=\emptyset$.
\end{lem}

\begin{proof}
Define $f: r_{n+1}[a,X]\ra 2$ by $f(t)=0$ if $t\in\mathcal{H}$, and $f(t)=1$ if $t\not\in\mathcal{H}$.
Then there is a $Y\in [\depth_X(a),X]$ satisfying Lemma \ref{lem.canon1ext}.
Since $f$ has only two values, neither (1)  nor (2) of   Lemma \ref{lem.canon1ext} can hold;
so $f$ must be constant on  $r_{n+1}[a,Y]$.
If $f$ is constantly $0$ on  $r_{n+1}[a,Y]$, then $ r_{n+1}[a,Y]\sse\mathcal{H}$;
otherwise, $f$ is constantly $1$ on  $r_{n+1}[a,Y]$, and 
$r_{n+1}[a,Y]\cap\mathcal{H}=\emptyset$.
\end{proof}

\begin{thm}\label{thm.E_2tRs}
$(\mathcal{E}_2,\le, r)$ is a topological Ramsey space.
\end{thm}

\begin{proof}
$(\mathcal{E}_2,\le,r)$ is a closed subspace of $\mathcal{AR}^{\om}$.
It is straightforward to check that \bf A.1 \rm and \bf A.2 \rm
 hold.
Lemma \ref{lem.A.3} shows that \bf A.3 \rm holds, and Lemma \ref{lem.A.4} shows that \bf A.4 \rm holds.
Thus, by the Abstract Ellentuck Theorem \ref{thm.AET},
$(\mathcal{E}_2,\le,r)$ is a topological Ramsey space.
\end{proof}

We now begin the inductive process of proving \bf A.4 \rm for $\mathcal{E}_k$, $k\ge 3$.
Let $k\ge 2$ and $1\le l<k$ be given.  
Let $U\sse \bW_k$ be given.
We say that {\em $U$ is isomorphic to a member of $\mathcal{E}_{k-l}$} if its structure is the same as $\bW_{k-l}$.
By this, we mean precisely the following:
Let $P=\{p<\om: \bW_k(\vec{i}_p)\in U\}$, and enumerate $P$ in increasing order as $P=\{p_m:m<\om\}$.
Let  the mapping $\theta:\{ \vec{i}_p:p\in P\}\ra \om^{\not\, \downarrow (k-l)}$ be given by $\theta(\vec{i}_{p_m})$ equals the $\prec$-$m$-th member of $\om^{\not\,\downarrow  (k-l)}$.
Then $\theta$ induces a tree isomorphism, respecting lexicographic order, from  the tree of all initial segments of members of 
 $\{ \vec{i}_p:p\in P\}$ to the tree of all initial segments of members of $(\om^{\not\, \downarrow \le (k-l)},\prec)$.
The next fact  generalizes Fact \ref{fact.dense}
to the $\mathcal{E}_k$, $k\ge 3$, and will be used in the inductive proof of \bf A.4 \rm for the rest of the spaces.

\begin{fact}\label{lem.thin}
Let  $k\ge 2$, 
$ l<k$,
$n\in N^k_l$,
$X\in\mathcal{E}_k$,
and $a\in \mathcal{AR}_n|X$  be given.
\begin{enumerate}
\item
Suppose $l\ge 1$ and $V\sse r_{n+1}[a,X]$ is such that 
$U:=\{b(\vec{i}_n):b\in V\}$
is isomorphic to a member of $\mathcal{E}_{k-l}$.
Then there is a $Y\in [a,X]$ such that $r_{n+1}[a,Y]\sse V$.
\item
Suppose $l=0$ and there is an infinite set $I\sse\{p\ge n: p\in N^k_0\}$ such that 
\begin{enumerate}
\item[(a)]
for all $p\ne q$ in $I$, 
$\vec{i}_p\re 1\ne \vec{i}_q\re 1$, and
\item[(b)]
 for each $p\in I$, there is a set $U_p\sse \{X(\vec{i}_q): q\in \om$ and $\vec{i}_q\re 1=\vec{i}_p\re 1\}$ such that $U_p$ is isomorphic to a member of $\mathcal{E}_{k-1}$.
\end{enumerate}
Then  there is a $Y\in[a,X]$ such that $r_{n+1}[a,Y]\sse\bigcup_{p\in I} U_p$.
\end{enumerate}
\end{fact}

\begin{proof}
To prove (1),
let $n,X,V,U$ satisfy the hypotheses.
Construct $Y\in[a,X]$ by starting with $a$, and 
choosing successively, for each $p\ge n$,  some $Y(\vec{i}_p)\in X$ such that whenever $\vec{i}_p\re l=\vec{i}_n\re l$, then $Y(\vec{i}_p)\in U$.

To prove (2), start with $a$.
Noting that $n\in N^k_{0}$,
take any $p\in I$ and 
choose $Y(\vec{i}_n)\in U_{p}$ such that 
$\max\pi_{l_n+1}(Y(\vec{i}_n))>\max a$.
Let $y_{n+1}=a_n\cup\{Y(\vec{i}_n)\}$.
Suppose we have chosen $y_m$, for $m\ge n$.
If $m\in N^k_0$, 
then take $p\in I$ such that  $\pi_1(X(\vec{i}_q))>\max y_m$ for each $X(\vec{i}_q)\in U_p$.
Take $Y(\vec{i}_m)$ to be any member of $U_p$.

If $m\in N^k_l$ for some $l>0$, then we have two cases.
Suppose $y_m(\vec{i}_m\re 1)\in\pi_1(a)$.
Then choose $Y(\vec{i}_m)$ to be any member of $X$  
such that,  
for $q<m$ such that $\vec{i}_m\re l=\vec{i}_q\re l$,
$\pi_l(Y(\vec{i}_m))=\pi_l(y_m(\vec{i}_q))$,
and $\max \pi_{l+1}(Y(\vec{i}_m))>\max y_m$.
Otherwise, $y_m(\vec{i}_m\re 1)\not\in\pi_1(a)$.
In this case, let $q<m$ such that $\vec{i}_m\re l=\vec{i}_q\re l$,
and let
 $p$ be such that $\pi_1(y_m(\vec{i}_q))\in\pi_1(U_p)$.
Then
 take $Y(\vec{i}_m)$ to be any member of $U_p$ 
such that the following hold:
$Y(\vec{i}_m\re l)=y_m(\vec{i}_q\re l)$, and $\max \pi_{l+1}(Y(\vec{i}_m))>\max y_m$.
Let $y_{m+1}=y_m\cup \{Y(\vec{i}_m)\}$.

Letting $Y=\bigcup_{m\ge n} y_m$,
we obtain a member of $\mathcal{E}_k$ which satisfies our claim. 
\end{proof}

The following lemma is proved by an induction scheme:
Given that $\mathcal{E}_k$ satisfies the Pigeonhole Principle,
 we  then prove that $\mathcal{E}_{k+1}$ satisfies the Pigeonhole Principle. 
In fact, one can prove this directly, but  induction streamlines the proof.

\begin{lem}\label{lem.pigeonhole_k}
For each $k\ge 3$, $\mathcal{E}_k$ satisfies \bf A.4\rm.
\end{lem}

\begin{proof}
By Lemma \ref{lem.A.4}, $\mathcal{E}_2$ satisfies \bf A.4\rm.
Now  assume that $k\ge 2$ and $\mathcal{E}_k$ satisfies \bf A.4\rm.
We will prove that $\mathcal{E}_{k+1}$ satisfies \bf A.4\rm.
Let $X\in\mathcal{E}_{k+1}$,  $a=r_n(X)$, and $\mathcal{O}\sse\mathcal{AR}_{n+1}$.
Let $l<k+1$ be such that $n\in N^{k+1}_l$.

Suppose $l\ge 1$ and let $k'=k+1-l$.
Letting $U$ denote $\{b(\vec{i}_n):b\in r_{n+1}[a,X]\}$, we note that $U$ is isomorphic to a member of $\mathcal{E}_{k'}$.
By the induction hypothesis, \bf A.4 \rm holds for $\mathcal{E}_{k'}$.
It follows that at least one of  $\{b(\vec{i}_n):b\in r_{n+1}[a,X]\cap\mathcal{O}\}$
or 
 $\{b(\vec{i}_n):b\in r_{n+1}[a,X]\setminus\mathcal{O}\}$
contains a set isomorphic to a member of $\mathcal{E}_{k'}$.
By Fact \ref{lem.thin} (1), 
there is a
 $Y\in[a,X]$ such that either  $r_{n+1}[a,Y]\sse \mathcal{O}$ or else 
$r_{n+1}[a,Y]\sse \mathcal{O}^c$.

Suppose now that  $l=0$.
Take $I$ to consist of those $p\ge n$ for which $\vec{i}_p\re 1>\vec{i}_q\re 1$ for all $q<p$.
Then $I$ is infinite.
Moreover, 
for each  $p\in I$, 
letting $I_p:=\{q\ge p:\vec{i}_q\re 1=\vec{i}_p\re 1\}$,
we have that 
$\{X(\vec{i}_q):q\in I_p\}$
 is isomorphic to a member of $\mathcal{E}_k$.
Thus, for each $p\in I$, at least one of $\{X(\vec{i}_q):q\in I_p \}\cap\{b(\vec{i}_n):b\in r_{n+1}[a,X]\cap \mathcal{O}\}$
or $\{X(\vec{i}_q)   :q\in I_p  \}\cap\{b(\vec{i}_n):b\in r_{n+1}[a,X]\cap\mathcal{O}^c\}$
contains a subset which is isomorphic to a member of $\mathcal{E}_k$.
Take one and call it $U_p$.
Thin $I$ to an infinite subset $I'$ for which either $U_p\sse
\{b(\vec{i}_n):b\in r_{n+1}[a,X]\cap
\mathcal{O}\}$ for all $p\in I'$, or else
$U_p\sse
\{b(\vec{i}_n):b\in r_{n+1}[a,X]\cap\mathcal{O}^c\}$ for all $p\in I'$.
By  Fact \ref{lem.thin} (2), 
there is a $Y\in[a,X]$ such that  $r_{n+1}[a,Y]\sse \bigcup_{p\in I'}U_p$.
Thus, $Y$ satisfies \bf A.4\rm.
\end{proof}

From Theorem \ref{thm.E_2tRs} and  Lemmas \ref{lem.A.3} and \ref{lem.pigeonhole_k}, we obtain the following theorem.

\begin{thm}
For each $2\le k<\om$, $(\mathcal{E}_k,\le r)$ is a topological Ramsey space.
\end{thm}


\section{Ramsey-classification theorems}\label{sec.rct}

In this section, we show that in each of the spaces $\mathcal{E}_k$, $k\ge 2$, the  analogue of the \Pudlak-\Rodl\ Theorem holds.
Precisely, we show  in Theorem \ref{thm.rc}
that   each equivalence relation on any given front on $\mathcal{E}_k$  is canonical when restricted to some member of $\mathcal{E}_k$.
  (See Definitions \ref{def.innerNW} and \ref{defn.canon}.)

Let $k\ge 2$ be fixed.
We begin with some basic notation, definitions and facts which will aid in the proofs.
From now on, we  routinely use the following abuse of notation.
\begin{notation}
For $X\in\mathcal{E}_k$ and $n<\om$, we shall use $X(n)$ to denote $X(\vec{i}_n)$.
\end{notation}

 We will often want to consider  the set of all $Y$ into which  a given finite approximation $s$ can be extended, even though $Y$ might not actually contain $s$. 
Thus, we define the following notation.

\begin{notation}\label{defn.ext}
Let  $s,t\in\mathcal{AR}$ and $X\in\mathcal{R}$.
Define
$\Ext(s)=\{Y\in\mathcal{R}:s \sse Y\}$,
and let 
$\Ext(s,t)$ denote $\Ext(s)\cap\Ext(t)$.
Define $\Ext(s,X)=\{Y\le X: Y\in\Ext(s)\}$,
and let $\Ext(s,t,X)$ denote $\Ext(s,X)\cap\Ext(t,X)$.

Define
$X/s=\{X(n): n<\om$ 
and   $\max X(n)>\max s\}$ and 
$a/s=\{a(n): n<|a|$ and $\max a(n)> \max s\}$.
Let $[s,X/t]$ denote $\{Y\in\mathcal{R}:s\sqsubset Y$ and $Y/s\sse X/t\}$.

Let $r_{n}[s,X/t]$ be $\{a\in\mathcal{AR}_{n}:a\sqsupseteq s$ and $a/s\sse X/t\}$.
For $m=|s|$,
let $r[s,X/t]$ denote  $\bigcup\{r_{n}[s,X/t]:  n\ge m\}$.
Let $\depth_X(s,t)$ denote $\max \{\depth_X(s),\depth_X(t)\}$.
\end{notation}

$\Ext(s,X)$ is the set of all $Y\le X$ {\em into} which $s$ can be extended to a member of $\mathcal{R}$.
Note that $Y\in\Ext(s,X)$  implies that there is a $Z\in\mathcal{R}$ such that $s\sqsubset Z$ and $Z/s\sse Y$.

\begin{fact}\label{fact.useful}
Suppose  $Y\le X\in \mathcal{E}_k$ and $c\sse c'$ in $\mathcal{AR}$ are given  with $\max c=\max c'$, $c\le_{\fin} Y$, and $c'\le_{\fin} X$.
Then there is a $Y'\in [c',X]$ such that
for any $s\le_{\mathrm{fin}} c$ and any $a\in r[s,Y'/c]$, 
$a/c\sse Y$.
\end{fact}

\begin{proof}
The proof is by the sort of  standard construction we have done in previous similar arguments.
Let $d=|c'|$, and let $r_d(Y')=c'$.
For $n\ge d$, having chosen $r_n(Y')$,
 let $l<k$ be such that $n\in  N^k_l$ and choose $Y'(n)$ as follows.
\begin{enumerate}
\item
If $l\ge 1$,
\begin{enumerate}
\item[(i)]
if $c'(\vec{i}_n\re l)\in\pi_l(c)$, then choose $Y'(n)\in Y$;
\item[(ii)]
if $c'(\vec{i}_n\re l)\not\in\pi_l(c)$, then choose $Y'(n)\in X$.
\end{enumerate}
\item
If $l=0$,
then choose $Y'(n)\in Y$.
\end{enumerate}
Then $Y'$ satisfies the conclusion.
\end{proof}

Recall Definition \ref{def.frontR1} of front on a topological Ramsey space from Section \ref{sec.reviewtRs}.

\begin{defn}\label{defn.mix}
Let $\mathcal{F}$ be a  front on $\mathcal{E}_k$ and let $f:\mathcal{F}\ra\om$.
Let $\hat{\mathcal{F}}=\{r_{n}(a):a\in\mathcal{F}$ and $n\le |a|\}$.
Suppose  $s,t\in\hat{\mathcal{F}}$ 
and $X\in\Ext(s,t)$.
We say that $X$ {\em separates} $s$ and $t$ if and only if
for all  $a\in\mathcal{F}\cap r[s,X/t]$ and $b\in \mathcal{F}\cap r[t,X/s]$, $f(a)\ne f(b)$.
We say that $X$ {\em mixes} $s$ and $t$ if and only if no $Y\in\Ext(s,t,X)$ separates $s$ and $t$.
We say that $X$ {\em decides} for $s$ and $t$ if and only if either $X$ mixes $s$ and $t$ or else $X$ separates $s$ and $t$.
\end{defn}

Note that mixing and separating of $s$ and $t$ only are defined for $X\in\Ext(s,t)$.
Though we could extend this to all $X$ in $\mathcal{E}_k$ by declaring $X$ to separate $s$ and $t$ whenever $X\not\in\Ext(s,t)$, this is unnecessary, as it will not be relevant to our construction.
Also note that $X\in\Ext(s,t)$  mixes $s$ and $t$ if and only if
 for each $Y\in\Ext(s,t,X)$, there are $a\in\mathcal{F}\cap r[s,Y/t]$ and $b\in\mathcal{F}\cap  r[t,Y/s]$ for which $f(a)=f(b)$.

\begin{fact}\label{fact.equivformsmix}
The following are equivalent for $X\in\Ext(s,t)$:
\begin{enumerate}
\item
$X$ mixes $s$ and $t$.
\item
For all $Y\in\Ext(s,t,X)$,
there are $a\in\mathcal{F}\cap r[s,Y/t]$ and $b\in\mathcal{F}\cap  r[t,Y/s]$ for which $f(a)=f(b)$.
\item
For all $Y\in[\depth_X(s,t),X]$, there are  $a\in\mathcal{F}\cap r[s,Y/t]$ and $b\in\mathcal{F}\cap  r[t,Y/s]$ for which $f(a)=f(b)$.
\end{enumerate}
\end{fact}

\begin{proof}
(1) $\Leftrightarrow$ (2) follows immediately from the definition of mixing.
(2) $\Rightarrow$ (3) is also immediate, since $[\depth_X(s,t),X]\sse \Ext(s,t,X)$.
To see that (3) implies (2),
let $Y\in \Ext(s,t,X)$ be given, and let $c=r_{\depth_X(s,t)}(Y)$.
By Fact \ref{fact.useful},
there is  a $Y'\in[\depth_X(s,t),X]$  such that  $r[s,Y'/t]\sse r[s,Y/t]$ and $r[t,Y'/s]\sse r[t,Y/s]$. 
By (3), there are $a\in \mathcal{F} \cap r[s,Y'/t]$ and $b\in \mathcal{F}\cap r[t,Y'/s]$ such that $f(a)=f(b)$.
By our choice of $Y'$, $a$ is in $r[s,Y/t]$ and $b$ is in $ r[t,Y/s]$. 
Thus, (2) holds.
\end{proof}

\begin{lem}[Transitivity of Mixing]\label{lem.transmix}
Suppose that $X$ mixes $s$ and $t$ and $X$ mixes $t$ and $u$.
Then $X$ mixes $s$ and $u$.
\end{lem}

\begin{proof}
Without loss of generality, we may assume that $\depth_X(u)\le\depth_X(s)$, and hence $[\depth_X(s),X]=[\depth_X(s,u),X]$.
Let $Y$ be any member of $[\depth_X(s),X]$.
We will show that there are $a\in\mathcal{F}\cap r[s,Y/u]$ and 
$c\in r[u,Y/s]$ such that $f(a)=f(c)$.
It then follows that 
 $X$ mixes $s$ and $u$.

Take $A\in [\depth_X(s,t),X]$ as follows:
If $\depth_X(t)\le\depth_X(s)$, then let $A=Y$.
If $\depth_X(t)>\depth_X(s)$,
then take $A$ so that  $r[u,A/t]\sse r[u,Y/t]$ and $r[s,A/t]\sse r[s,Y/t]$.
This is possible by Fact \ref{fact.useful}.

Define 
\begin{equation}
\mathcal{H}=\{b\in\mathcal{F}\cap r[t,A/s]:\exists a\in \mathcal{F}\cap r[s,Y/t]( f(a)=f(b))\}.
\end{equation}
Define $\mathcal{X}=\bigcup\{[b,A]:b\in\mathcal{H}\}$.
Then $\mathcal{X}$ is an open set, so by the Abstract Ellentuck Theorem,
there is a $B\in[\depth_X(s,t),A]$ such that either $[\depth_X(s,t),B]\sse\mathcal{X}$ or else
$[\depth_X(s,t),B]\cap\mathcal{X}=\emptyset$.

If $[\depth_X(s,t),B]\cap\mathcal{X}=\emptyset$, then for each $b\in \mathcal{F}\cap r[t,B/s]$ and each $a\in\mathcal{F}\cap r[s,Y/t]$, we have that $f(a)\ne f(b)$.
Since $r[s,B/t]\sse r[s,A/t]\sse r[s,Y/t]$,
we have that $B$ separates $s$ and $t$, a contradiction.
Thus, $[\depth_X(s,t),B]\sse\mathcal{X}$.

Since $\depth_X(u)\le\depth_X(s)$, it follows that $B\in \Ext(t,u,X)$; so $B$ mixes $t$ and $u$.
Take $b\in\mathcal{F}\cap r[t,B/u]$ and $c\in \mathcal{F}\cap r[u, B/t]$ such that $f(b)=f(c)$.
Since $[\depth_X(s,t),B]\sse\mathcal{X}$, it follows that
$\mathcal{F}\cap r[t,B/s]\sse \mathcal{H}$.
Thus, there is an $a\in\mathcal{F}\cap r[s,Y/t]$ such that $f(a)=f(b)$.
Hence, $f(a)=f(c)$.
Note that $a\in r[s,Y/u]$ trivially, since $\depth_X(u)\le\depth_X(s)$.
Moreover,
$c\in r[u,Y/s]$:
To see this, note first that
$r[u,B/t]\sse r[u,A/t]$.
Secondly, $A=Y$ if $\depth_X(t)\le\depth_X(s)$, 
and $r[u,A/t]\sse r[u,Y/t]$ 
if $\depth_X(t)>\depth_X(s)$.
Therefore, $Y$ mixes $s$ and $u$.
\end{proof}

Next, we define the notion of a hereditary property, and give   a general lemma about fusion  to obtain a member of $\mathcal{E}_k$ on which a hereditary property holds.

\begin{defn}\label{hered}
A property  $P(s,X)$ defined on $\mathcal{AR}\times\mathcal{R}$
 is  {\em hereditary} if whenever $X\in\Ext(s)$ and $P(s,X)$ holds, then also $P(s,Y)$ holds for all $Y\in[\depth_X(s),X]$.
Similarly, a property $P(s,t,X)$ defined on $\mathcal{AR}\times\mathcal{AR}\times\mathcal{R}$
 is {\em hereditary} if whenever $P(s,t,X)$ holds, then also $P(s,t,Y)$ holds for all $Y\in [\depth_X(s,t),X]$.
\end{defn}

\begin{lem}\label{lem.hered}
Let $P(\cdot,\cdot)$ be a hereditary property on $\mathcal{AR}\times\mathcal{R}$.
If whenever $X\in\Ext(s)$ there is a $Y\in[\depth_X(s),X]$ such that $P(s,Y)$,
then for each $Z\in\mathcal{R}$, there is a $Z'\le Z$ such that for all $s\in\mathcal{AR}|Z'$,
$P(s,Z')$ holds.

Likewise, suppose  $P(\cdot,\cdot,\cdot)$ is  a hereditary property on $\mathcal{AR}\times\mathcal{AR}\times\mathcal{R}$.
If whenever $X\in \Ext(s,t)$ there is a $Y\in[\depth_X(s,t),X]$ such that $P(s,t,Y)$,
then for each $Z\in\mathcal{R}$, there is a $Z'\le Z$ such that for all $s,t\in\mathcal{AR}|Z'$,
$P(s,t,Z')$ holds.
\end{lem}

The proof of Lemma \ref{lem.hered} is straightforward;  being very similar to that of  Lemma 4.6 in \cite{Dobrinen/Todorcevic14}, we omit it.

\begin{lem}\label{lem.decide}
Given any front $\mathcal{F}$ and function $f:\mathcal{F}\ra\om$,
there is an $X\in\mathcal{E}_k$ such that for all $s,t\in\hat{\mathcal{F}}|X$,
$X$ decides $s$ and $t$.
\end{lem}

Lemma \ref{lem.decide} follows immediately from Lemma \ref{lem.hered}  and the fact that mixing and separating are hereditary properties.

For $a\in\mathcal{AR}$ and $\vec{l}\in (k+1)^{|a|}$,
we shall let $\pi_{\vec{l}}(a)$ denote $\{\pi_{l_m}(a(m)):m<|a|\}$.

\begin{defn}\label{def.innerNW}
A map $\vp$ on a front $\mathcal{F}\sse \mathcal{AR}$ is called
\begin{enumerate}
\item
{\em inner} if for each $a\in \mathcal{F}$, 
$\vp(a)=\pi_{\vec{l}}(a)$, for some $\vec{l}\in (k+1)^{|a|}$.
\item
{\em Nash-Williams} if for all pairs $a,b\in \mathcal{F}$,
whenever 
 $\vp(b)=\pi_{\vec{l}}(b)$ and 
there is some $n\le |b|$ such that $\vp(a)=\pi_{(l_0,\dots,l_{n-1})}(r_n(b))$,
then $\vp(a)=\vp(b)$.
\item 
{\em irreducible} if it is inner and Nash-Williams.
\end{enumerate}
\end{defn}

\begin{defn}[Canonical equivalence relations on a front]\label{defn.canon}
Let $\mathcal{F}$ be a front on $\mathcal{E}_k$.
An equivalence relation $R$ on $\mathcal{F}$ is {\em canonical} if  and only if
there is an irreducible map $\vp$ canonizing $R$ on $\mathcal{F}$,
meaning that for all $a,b\in\mathcal{F}$, 
$a\, R\, a \llra \vp(a)=\vp(b)$.
\end{defn}

We shall show in Theorem \ref{thm.irred} (to be proved after  Theorem \ref{thm.rc})
that, similarly to the Ellentuck space, irreducible maps on $\mathcal{E}_k$ 
are unique in the following sense.

\begin{thm}\label{thm.irred}
Let  $R$ be an equivalence relation on some front $\mathcal{F}$ on $\mathcal{E}_k$.
Suppose $\vp$ and $\vp'$ are irreducible maps  canonizing $R$.
Then there is an $A\in\mathcal{E}_k$ such that 
for each $a\in\mathcal{F}|A$,
$\vp(a)=\vp'(a)$.
\end{thm}

\begin{defn}\label{def.1extcanon}
For each  pair $X,Y\in\mathcal{E}_k$, $m,n<\om$,  and $l\le k$, define
\begin{equation}
X(m)\, E_l\, Y(n)\llra \pi_l(X(m))=\pi_l(Y(m)).
\end{equation}
\end{defn}

Note that $X(m)\, E_0\, Y(n)$ for all $X,Y$ and $m,n$,
and $X(m)\, E_k\, Y(n)$ if and only if $X(m)=Y(n)$.
Let $\mathfrak{E}_k$ denote $\{E_l:l\le k\}$,
the set of {\em canonical equivalence relations} on $1$-extensions.

We now prove  the  Ramsey-classification theorem for equivalence relations on fronts.
The proof generally follows the same form as that  of Theorem 4.14 in \cite{Dobrinen/Todorcevic14}, the modifications either being proved or  pointed out.
One of the main differences is that, in our spaces $\mathcal{E}_k$, for any given  $s\le_{\fin} X$  there will be many $Y\le X$ such that $s$ cannot be extended into $Y$, and this has to be handled with care.
The other main difference is the type of  inner Nash-Williams maps for our spaces here necessitate quite different proofs of Claims \ref{claim.4.15} and \ref{claim.phiknows} from their analagous statements in \cite{Dobrinen/Todorcevic14}.
Finally, analogously to the Ellentuck space, the canonical equivalence relations are given by irreducible maps which are unique in the sense of Theorem \ref{thm.irred}.
This was not the case for the  topological Ramsey spaces in \cite{Dobrinen/Todorcevic14},
\cite{Dobrinen/Todorcevic15} and \cite{Dobrinen/Mijares/Trujillo14}, 
which can have different inner Nash-Williams maps canonizing the same equivalence relation; for those spaces, we showed that the right canonical map  is the maximal one.

\begin{thm}[Ramsey-classification Theorem]\label{thm.rc}
Let $2\le k<\om$ be fixed.
Given $A\in\mathcal{E}_k$ and   an equivalence relation $R$ on a front $\mathcal{F}$ on $A$,
there is a member $B\le A$ such that $R$ restricted to $\mathcal{F}|B$ is canonical.
\end{thm}

\begin{proof}
By Lemma \ref{lem.decide} and shrinking $A$ if necessary, we may assume  that for all $s,t\in \hat{\mathcal{F}}|A$,
$A$ decides for $s$ and $t$.
For  $n<\om$, $s\in\mathcal{AR}_{n}$,  $X\in\Ext(s)$, and $E\in\mathfrak{E}_k$,
we shall  say that $X$ $E$-{\em mixes} $s$ if and only if
for all $a,b\in r_{n+1}[s,X]$,
\begin{equation}
X\mathrm{\ mixes\ } a \mathrm{\ and\ } b\ \ \llra \ \  a(n)\, E\, b(n).
\end{equation}

\begin{claim}\label{claim.4.15}
There is an $A'\le A$ such that for each $s\in(\hat{\mathcal{F}}\setminus\mathcal{F})|A'$, 
letting $n=|s|$, the following holds:
There is a canonical equivalence relation $E_s\in\mathfrak{E}_k$ such that 
for all $a,b\in r_{n+1}[s,A']$,
$B$ mixes $a$ and $b$ if and only if $a(n)\, E_s\, b(n)$.
Moreover,  $n\in N^k_l$ implies $E_s$ cannot be $E_j$ for any $1\le j \le l $.
\end{claim}

\begin{proof}
Let $X\le A$ be given and $s\in(\hat{\mathcal{F}}\setminus\mathcal{F})|A$.
Let $n=|s|$ and $l< k$ be such that $n\in N^k_l$.
We will show that there is a $Y\in [\depth_X(s),X]$ and 
either $j=0$ or else a $l<j\le k$ such that
 for each 
$a,b\in r_{n+1}[s,Y]$,
$Y$ mixes $a$ and $b$ if and only if $a(n)\, E_j\, b(n)$.
The Claim will then immediately follow from Lemma \ref{lem.hered}.

First, let $m>n$ be least such that for any $a\in r_{m+1}[s,X]$, 
$\pi_{l}(a(m))=\pi_{l}(a(n))$ but $\pi_{l+1}(a(m))> \pi_{l+1}(a(n))$.
Define
\begin{equation}
\mathcal{H}_{l+1}=\{a\in r_{m+1}[s,X]: A \mathrm{\ mixes\ } s\cup a(m)\mathrm{\ and\ } s\cup a(n)\}.
\end{equation}
By the Abstract Nash-Williams Theorem, there is a $Y_{l+1}\in [s,X]$ such that either $r_{m+1}[s,Y_{l+1}]\sse\mathcal{H}_{l+1}$, or else $r_{m+1}[s,Y_{l+1}]\cap\mathcal{H}_{l+1}=\emptyset$.
If $r_{m+1}[s,Y_{l+1}]\sse\mathcal{H}_{l+1}$, then every pair of $1$-extensions of $s$ into $Y_{l+1}$ is mixed by $A$; 
hence,  $E_s\re r_{n+1}[s,Y_{l+1}]$ is given by $E_0$.
In this case, let $A'=Y_{l+1}$.
Otherwise,  $r_{m+1}[s,Y_{l+1}]\cap\mathcal{H}_{l+1}=\emptyset$,
 so every pair 
of $1$-extensions of $s$ into $Y_{l+1}$ which differ on level $l+1$   is separated by $A$.

For the induction step, for $l+1\le j< k$,
suppose that  $Y_{j}$ is given and every pair of $1$-extensions of $s$ into $Y_{j}$ which differ on level $j$ is separated by $A$.
Let $m>n$ be least such that for any $a\in r_{m+1}[s,X]$, 
$\pi_{j}(a(m))=\pi_{j}(a(n))$ but $\pi_{j+1}(a(m))\ne \pi_{j+1}(a(n))$.
Define
\begin{equation}
\mathcal{H}_{j+1}=\{a\in r_{m+1}[s,X]: A \mathrm{\ mixes\ } s\cup a(m)\mathrm{\ and\ } s\cup a(n)\}.
\end{equation}
By the Abstract Nash-Williams Theorem, there is a $Y_{j+1}\in [s,Y_{j}]$ such that either $r_{m+1}[s,Y_{j+1}]\sse\mathcal{H}_{j+1}$, or else $r_{m+1}[s,Y_{j+1}]\cap\mathcal{H}_{j+1}=\emptyset$.
If $r_{m+1}[s,Y_{j+1}]\sse\mathcal{H}_{j+1}$, then every pair of $1$-extensions of $s$ into $Y_{j+1}$ is mixed by $A$; 
hence,  $E_s\re r_{n+1}[s,Y_{j+1}]=E_{j+1}$.
In this case, let $A'=Y_{j+1}$.

Otherwise,  $r_{m+1}[s,Y_{j+1}]\cap\mathcal{H}_{j+1}=\emptyset$, so every pair of $1$-extensions of $s$ into $Y_{j+1}$ which differ on level $j+1$ is separated by $A$.
If $j+1<k$, continue the induction scheme.
If the induction process 
terminates at some stage $j+1<k$, then letting $A'=Y_{j+1}$ satisfies the claim.
Otherwise,  the induction
does not terminate before $j+1=k$,
 in which case  $E_s\re r_{n+1}[s,Y_k]=E_k$ and we let $A'=Y_k$.

The above arguments show that $E_s\re r_{n+1}[a,A']$ is given by $E_j$, where either $j=0$ or else $l<j\le k$.
\end{proof}

For $s\in\mathcal{AR}_{n}|A'$, 
let $E_s$ denote the canonical equivalence relation for mixing 1-extensions of $s$ in  $r_{n+1}[s,A']$ from Claim \ref{claim.4.15}, and let $\pi_s$ denote the projection map on $\{t(n):t\in r_{n+1}[s,A']\}$  determined by  $E_s$.
Thus, for $a\in r_{n+1}[s,A']$, if $n\in N^k_l$, then 
\begin{equation}
\pi_s(a(n))=\emptyset \llra E_s\, =\, E_0,
\end{equation}
and for  $l<j\le k$,  
\begin{equation}
\pi_s(a(n))=\pi_j(a(n)) \llra E_s\, =\, E_{j}.
\end{equation}

\begin{defn}\label{def.proj}
For $t\in \hat{\mathcal{F}}|A'$,
define 
\begin{equation}
\vp(t)=\{\pi_{s}(t(m)): 
s\sqsubset t\mathrm{\ and\ } m=|s|\}.
\end{equation}
\end{defn}
It follows immediately from the definition that $\vp$ is an inner  map on $\mathcal{F}|A'$.

The next fact is straightforward, its proof  so closely resembling that of Claim 4.17 in \cite{Dobrinen/Todorcevic14} that we do not include it here.

\begin{fact}\label{claim.4.17}
Suppose $s\in(\hat{\mathcal{F}}\setminus\mathcal{F})|A'$ and $t\in\hat{\mathcal{F}}|A'$.
\begin{enumerate}
\item
Suppose $s\in\mathcal{AR}_{n}|A'$  and  $a,b\in r_{n+1}[s,A']$.
If $A'$ mixes $a$ and $t$ and $A'$ mixes $b$ and $t$, then $a(n)\, E_s\,  b(n)$.
\item
If $s\sqsubset t$ and $\vp(s)=\vp(t)$,
then $A'$ mixes $s$ and $t$.
\end{enumerate}
\end{fact}

The next lemma is the crux of the proof of the theorem.

\begin{claim}\label{claim.phiknows}
There is a $B\le A'$ such that 
for all $s,t\in(\hat{\mathcal{F}}\setminus\mathcal{F})|B$  which are mixed by $B$, the following holds:
For all $a\in r_{|s|+1}[s,B/t]$ and $b\in r_{|t|+1}[t,B/s]$,
$B$ mixes $a$ and $b$ if and only if $\pi_s(a(|s|))=\pi_t(b(|t|))$.
\end{claim}

\begin{proof}
We will show that for all pairs $s,t\in (\hat{\mathcal{F}}\setminus\mathcal{F})|A'$
which are  mixed by $A'$,
for each $X\in\Ext(s,t,A')$,
there is a $Y\in[\depth_X(s,t),X]$ such that 
for all  $a\in r_{|s|+1}[s,Y/t]$ and $b\in r_{|t|+1}[t,Y/s]$,
$A'$ mixes $a$ and $b$ if and only of $\vp_s(a(|s|))=\vp_t(b(|t|))$.
The conclusion will then follow from  Fact \ref{fact.equivformsmix} and Lemma \ref{lem.hered}.

Suppose
$s,t\in(\hat{\mathcal{F}}\setminus\mathcal{F})|A'$
are mixed by $A'$.
Let $m=|s|$, $n=|t|$,  $X\in\Ext(s,t,A')$,  
and  $d=\depth_X(s,t)$.

\begin{subclaim}\label{subclaim.bothE_0}
$E_s=E_0$ if and only if $E_t= E_0$.
\end{subclaim}

\begin{proof}
Suppose toward a contradiction that $E_s=E_0$ but $E_t\ne E_0$.
Let $l<k$ be such that $n\in N^k_l$.
Then $E_t\ne E_0$ implies  $E_t=E_p$ for some $l<p\le k$, by Claim \ref{claim.4.15}.
Fact \ref{claim.4.17} (1) implies 
that there is at most one $E_t$ equivalence class of 1-extensions $b$ of $t$ for which $b$ is mixed with each $1$-extension of $s$.
If each $b\in r_{n+1}[t,X/s]$ is not mixed with any $a\in r_{m+1}[s,X/t]$,
then $X$ separates $s$ and $t$, a contradiction.
So, suppose  $b\in r_{n+1}[t,X/s]$ is mixed with some $a\in r_{m+1}[s,X/t]$.
By Fact \ref{claim.4.17} (2), all 1-extensions $a,a'$ of $s$ are mixed.
Hence, $X$ mixes $b$ with every $a\in r_{m+1}[s,X/t]$.
Take $Y\in [d,X]$ such that, for the $j<k$ such that  $d\in N^k_j$,
$\max\pi_{j+1}(Y(d))>\max b$.
Then for each $b'\in r_{n+1}[t,Y/s]$, $\pi_p(b')>\pi_p(b)$, so $b'  \hskip-.05in\not\hskip-.07in E_t\, b$.
Hence, $b'$
is separated from each $a\in r_{m+1}[s,Y/t]$. 
But this contradicts that $X$ mixes $s$ and $t$.
Therefore, $E_t$ must also be $E_0$.
\end{proof}

Suppose
 both $E_s$ and $E_t$ are $E_0$.
Then  for all 
 $a\in r_{m+1}[s,X/t]$ and 
$b\in r_{n+1}[t,X/s]$,
$A'$ mixes $a$ and $b$, by Fact \ref{claim.4.17} (2) and transitivity of mixing. 
At the same time, $\pi_s(a(m))=\pi_t(b(n))=\emptyset$.
In this case simply let $Y=X$.

\begin{subclaim}\label{subclaim.miximpliesp=q}
Assume that $E_s\ne E_0$ and $E_t\ne E_0$.
Let $p,q$ be the numbers such that $m\in N^k_{p}$ and $n\in N^k_{q}$.
If $p\ne q$, then $A$ separates $s$ and $t$.
\end{subclaim}

\begin{proof}
Since both $E_s$ and $E_t$ are not $E_0$,
there are some $j,l$ such that $p<j\le k$, $q<l\le k$, $E_s=E_j$, and $E_t=E_l$.
Suppose without loss of generality that $q<p$.
Since $m\in N^k_{p}$, it follows that for each $a\in r_{m+1}[s,A'/t]$, 
$\pi_p(a(m))\in\pi_p(s)$.
Furthermore, $\max\pi_p(a(m))<\max r_d(A')$,
and $\max\pi_{p+1}(a(m))>\max(r_d(A'))$,
where $d=\depth_{A'}(s,t)$.
Since $n\in N^k_{q}$, it follows that for each $b\in r_{n+1}[t,A'/s]$,
$\max\pi_{q+1}(b(n))>\max r_d(A')$.
Since $q<p$, 
every pair of  1-extensions of $s$ have the same $\pi_{q+1}$ value.
On the other hand,
every pair of  1-extensions of $t$ with different $\pi_{q+1}$ values  are separated, since $l\ge q+1$.
In particular,
$a(m)$ is never equal to $b(n)$,
for all $a\in r_{m+1}[s,A'/t]$ and $b\in r_{n+1}[t,A'/s]$.

Let $n'>d$ be minimal in $N^k_{q}$ such that 
there is an $m'\in N^k_{p}$ with $d\le m'<n'$,
 and such that
for each $c\in r_{n'+1}[d,A']$, both $s\cup c(m')\in r_{m+1}[s,A'/t]$ and 
$t\cup c(n')\in r_{n+1}[t,A'/s]$.
Define 
\begin{equation}
\mathcal{H}'=\{c\in r_{n'+1}[d,A']:A\mathrm{\ mixes\ } s\cup c(m')\mathrm{\ and\ } t\cup c(n')\}.
\end{equation}
Let $m''>d$ be minimal in $N^k_{p}$ such that there is an $n''\in N^k_{q}$ with $d\le n''<m''$,
 and such that
for each $c\in r_{m''+1}[d,A']$, both $s\cup c(m'')\in r_{m+1}[s,A'/t]$ and 
$t\cup c(n'')\in r_{n+1}[t,A'/s]$.
Define 
\begin{equation}
\mathcal{H}''=\{c\in r_{m''+1}[d,A']:A\mathrm{\ mixes\ } s\cup c(m'')\mathrm{\ and\ } t\cup c(n'')\}.
\end{equation}
Take $Y\in[d,A']$ homogeneous for both $\mathcal{H}'$ and $\mathcal{H}''$.

If $r_{n'+1}[d,Y]\sse\mathcal{H}'$,
then there are  two different 1-extensions  of $t$ in $Y$ above $s$ which are not $E_t$-related, yet are both mixed with the same extension of $s$, a contradiction, since mixing is transitive.
Similarly, if $r_{n'+1}[d,Y]\sse\mathcal{H}''$, we obtain a contradiction.
Thus, both $r_{n'+1}[d,Y]\cap\mathcal{H}'=\emptyset$ and 
$r_{m''+1}[d,Y]\cap\mathcal{H}''=\emptyset$; hence $Y$ separates $s$ and $t$.

Similarly, if $p<q$, we conclude that there is a $Y\in [d,A']$ which separates $s$ and $t$.
Since $A$ already decides $s$ and $t$, it follows that $A$ separates $s$ and $t$.
\end{proof}

By Subclaim \ref{subclaim.miximpliesp=q},  $s$ and $t$ being mixed by $A$ implies that
$p$ and $q$ must be equal.
Further, $s$ and $t$ mixed by $A$ also implies  $j$ must equal $l$.
To see this, supposing that $j<l$,  let $d\le m'<n'$ be such that $m'\in N^k_{p}$ and $n'\in N^k_{q}$,  and such that for each
$c\in r_{n'+1}[d,A']$, both $s\cup c(m')\in r_{m+1}[s,A'/t]$ and 
$t\cup c(n')\in r_{n+1}[t,A'/s]$.
Let
\begin{equation}
\mathcal{H}=\{c\in r_{n'+1}[d,A']:A \mathrm{\ mixes\ }s\cup c(m')\mathrm{\ and\ } t\cup c(n')\}.
\end{equation}
Then taking $Y\in [d,A']$ homogenous for $\mathcal{H}$, we find that $Y$ must separate these extensions of $s$ and $t$.
Likewise, for $n'<m'$.
Similarly, if $l<j$, we find a $Y\in[d,A']$ which separates $s$ and $t$, a contradiction.
Therefore, $j=l$.

\begin{subclaim}\label{subclaim.mixiff=}
There is a $Y\in [d,X]$ such that for all $a\in r_{m+1}[s,Y/t]$ and $b\in r_{n+1}[t,Y/s]$,
$Y$ mixes $a$ and $b$ if and only if $\pi_s(a(m))=\pi_t(b(n))$.
\end{subclaim}

\begin{proof}
We have already shown that $A$ mixing $s$ and $t$  implies that $p=q$ and $j=l$.
For each pair  $j+1\le j'\le k$ and $\rho\in\{<,=,>\}$, 
choose   minimal
$m',n'\in N^k_p$ such that $m',n'>d$,
$m'\, \rho\,  n'$,
and
for each
$c\in r_{\max(m',n')+1}[d,A']$, both $s\cup c(m')\in r_{m+1}[s,A'/t]$ and 
$t\cup c(n')\in r_{n+1}[t,A'/s]$.
For each such quadruple  $(j',m',n',\rho)$,
 let $\mathcal{H}_{(j',m',n',\rho)}$ denote the set of all $c\in r_{\max(m',n')+1}[d,X]$ such that $A$ mixes $s\cup c(m')$ and $t\cup c(n')$.
Take a $Y\in[d,X]$ which is homogeneous for all these sets.  Since there are only finitely many such quadruples, such a $Y$ exists.

Let $a\in r_{m+1}[s,Y/t]$ and $b\in r_{n+1}[t,Y/s]$.
Let $m',n'$ be least such that there is a $c\in r_{\max(m',n')+1}[d,Y]$ such that $a(m)=c(m')$ and $b(n)=c(n')$.
Let $\rho\in\{<,=,>\}$ be the relation such that $m'\, \rho\,  n'$.
If $\pi_j(a(m))\ne \pi_j(b(n))$, then $\rho\ne =$.
If $Y$ mixes $a$ and $b$, then in the case that $\rho$ is $<$, 
 there are $c,c'\in r_{n'+1}[d,Y]$ 
such that $c(m')=c'(m')$ but $\pi_j(c(n'))\ne \pi_j(c'(n'))$.
If $r_{n'+1}[d,Y]\sse \mathcal{H}_{(j',m',n',\rho)}$, then by transitivity of mixing,
 $Y$ mixes $s\cup c(n')$ and $s\cup c'(n')$.
But this contradicts Claim \ref{claim.4.15},
since $\pi_j(c(n'))\ne \pi_j(c'(n'))$.
Therefore, it must be the case that $r_{n'+1}[d,Y]\cap\mathcal{H}_{(j',m',n',\rho)}=\emptyset$, and hence,
$Y$ separates $a$ and $b$.
Likewise, if $\rho$ is $>$, we find that $Y$ separates $a$ and $b$.

Since by our assumption $s$ and $t$ are mixed by $A$ and $Y\le A$, 
$s$ and $t$ are mixed by $Y$.
Thus,
there must be some 1-extensions of $s$ and $t$ in $Y$ which are mixed by $Y$.
The only option left is that $Y$ mixes $a$ and $b$ when $\pi_l(a(m))=\pi_l(b(n))$.
Thus, 
$a$ and $b$ are mixed by $Y$ if and only if $\pi_l(a(m))=\pi_l(b(n))$.
\end{proof}

By Subclaim \ref{subclaim.mixiff=} and Lemma \ref{lem.hered},
the Claim holds.
\end{proof}

The next claim and its proof are  similar to Claim 4.19 in \cite{Dobrinen/Todorcevic14}.
We include it, as the modifications might not be obvious to the reader referring to \cite{Dobrinen/Todorcevic14}.

\begin{claim}\label{claim4.19}
For all $s,t\in \hat{\mathcal{F}}|B$, if $\vp(s)=\vp(t)$,
then $B$ mixes $s$ and $t$.
\end{claim}

\begin{proof}
Suppose that $\vp(s)=\vp(t)$.
By the definition of $\vp$, it follows that for all $n$, 
$\vp(s)\cap 
\{\pi_l(B(m)):  l\le k,\ m<n\}
=\vp(t)\cap \{\pi_l(B(m)):  l\le k,\ m<n\}$.
We show by induction that $B$ mixes $s\cap r_{n}(B)$ and $t\cap r_{n}(B)$ for all $n$. 
For the basis, $s\cap r_{0}(B)=t\cap r_{0}(B)=\emptyset$,
so $B$ trivially mixes $s\cap r_{0}(B)$ and $t\cap r_{0}(B)$.

Suppose that $B$ mixes $s\cap r_{n}(B)$ and $t\cap r_{n}(B)$.
Let $i,j$  be such that $r_{i}(s)=s\cap r_{n}(B)$ and $r_{j}(t)=t\cap r_{n}(B)$.
If $s\cap B(n)=t\cap B(n)=\emptyset$, then $B$ mixes $s\cap r_{n+1}(B)$ and $t\cap r_{n+1}(B)$.
If  $s\cap B(n)\ne\emptyset$, then $s(i)=s\cap B(n)$.
If  $t\cap B(n)=\emptyset$, then  
 $E_{r_{i}(s)}$ must be $E_{\emptyset}$, since $\vp(s)=\vp(t)$.
Then $B$ mixes $r_{i}(s)$ and $r_{i+1}(s)$, which equals $s\cap r_{n+1}(B)$.
Thus, $B$ mixes $s\cap r_{n+1}(B)$ and $t\cap r_{n+1}(B)$, 
since $t\cap r_{n+1}(B)=t\cap r_{n}(B)$.
Otherwise, $t\cap B(n)\ne\emptyset$, in which case $t(j)=t\cap B(n)$.
Since $\vp$ is inner and $\vp(s)=\vp(t)$,
there is an $l\le k$ such that 
 $\vp(s)\cap\{\pi_{l'}(B(n)):l'\le k\}
=\vp(t)\cap
\{\pi_{l'}(B(n)):l'\le k\}
=\pi_l(B(n))$.
This implies that $\vp_{r_{i}(s)}(s(i))=\vp_{r_{j}(t)}(t(j))$.
By Claim \ref{claim.phiknows},  $B$ mixes $r_{i+1}(s)=s\cap r_{n+1}(B)$ and $r_{j+1}(t)= t\cap r_{n+1}(B)$.
The case when 
$s\cap B(n)=\emptyset$
and
$t\cap B(n)\ne\emptyset$
is similar.
Thus, by induction, we find that $B$ mixes $s$ and $t$.
\end{proof}

Claim \ref{claim4.19} and  Fact \ref{claim.4.17} (1) imply that 
$\vp$ is a Nash-Williams function on $\mathcal{F}|B$.
As the  proof is  almost identical to that of Claim 4.20 in \cite{Dobrinen/Todorcevic14}, we omit it.
We finally obtain that for all  $s,t\in\mathcal{F}|B$, if $f(s)=f(t)$, then $\vp(s)=\vp(t)$, 
by a proof similar to that of 
 Claim 4.21 in \cite{Dobrinen/Todorcevic14}.

This concludes the proof of the Ramsey-classification theorem.
\end{proof}

We now prove that irreducible maps are unique, up to restriction below some member of the space.
\vskip.1in

\noindent \it Proof of Theorem \ref{thm.irred}. \rm
Let  $R$ be an equivalence relation on some front $\mathcal{F}$ on $\mathcal{E}_k$, and let $A\in\mathcal{E}_k$ be such that  the irreducible map $\vp$ from the proof of Theorem \ref{thm.rc} canonizes $R$ on $\mathcal{F}|A$.
Let  $\vp'$ be  any  irreducible map canonizing $R$ on $\mathcal{F}$.
Then $\vp'$ is a map from $\mathcal{F}$ into an infinite set, namely $[\widehat{\bW}_k]^{<\om}$.
Applying the proof of Theorem \ref{thm.rc} to $\vp'$,
we find a $B\le A$ such that for each 
 $t\in\mathcal{F}|B$  and $n<|t|$,
there is a sequence $\lgl l_{t,0},\dots,l_{t,|t|-1}\rgl$ such that  for each $n<|t|$, $\vp'(t)\cap\widehat{t(n)}=\pi_{l_n}(t(n))$, and
$\vp'(t)=\{\pi_{l_i}(t(i)):i<|t|\}$.
Now if 
 $\vp(t)\ne\vp'(t)$ for some  $t\in\mathcal{F}|B$, then 
there is some $n<|t|$ for which $\vp(t)\cap\widehat{t(n)}\ne \vp'(t)\cap\widehat{t(n)}$.
Let $m$ denote the integer less than $k$ such that $\pi_{r_n(t)}=\pi_m$.
If $l_{t,n}<m$, then there are $s,s'\in\mathcal{F}|B$ such that $s,s'\sqsupset r_n(t)$ and
$\vp(s)=\vp(s')$,
 but $\pi_{m}(s(n))\ne \pi_{m}(s'(n))$ and hence $\vp'(s)\ne\vp'(s')$.
This contradicts that $\vp$ and $\vp'$ canonize the same equivalence relation.
Likewise, if $m<l_{t,n}$, we obtain a contradiction.
Therefore, $\vp(t)$ must equal $\vp'(t)$ for all $t\in
\mathcal{F}|B$.
\hfill $\square$
\vskip.1in

As a corollary of Theorem \ref{thm.rc}, we obtain the following canonization theorem for the finite rank fronts $\mathcal{AR}_n$, the case of $n=1$ providing a higher order analogue of the \Erdos-Rado Theorem (see \cite{Erdos/Rado50}) for the Ellentuck space.

\begin{cor}\label{cor.er}
Let $k\ge 2$, $n\ge 1$, and $R$ be an equivalence relation on $\mathcal{AR}_n$ on the space $\mathcal{E}_k$.
Then there is   an  $A\in\mathcal{E}_k$ and there are $l_i\le k$  ($i<n$) such that  for each pair  $a,b\in\mathcal{AR}_n|A$,
$a\, R\, b$ if and only if  for each $i<n$, $\pi_{l_i}(a(i))=\pi_{l_i}(b(i))$.
Moreover, for each $i<n$, if $m$ is such that $i\in N^k_m$, then either $l_i=0$ or else $m+1\le l_i\le k$.
\end{cor}


\section{Basic  cofinal maps from the generic ultrafilters}\label{sec.basic}

In Theorem 20  in \cite{Dobrinen/Todorcevic11},
it was proved that every monotone cofinal map from a p-point into another ultrafilter is actually continuous, after restricting below some member of the p-point.
This property of p-points was  key in 
  \cite{Raghavan/Todorcevic12}, \cite{Dobrinen/Todorcevic14}, \cite{Dobrinen/Todorcevic15}, and \cite{Dobrinen/Mijares/Trujillo14} to 
pulling out a Rudin-Keisler map on a front from a cofinal map on an ultrafilter, thereby, along with the appropriate Ramsey-classification theorem, allowing for a fine analysis of initial Tukey structures in terms of Rudin-Keisler isomorphism types.
Although the generic ultrafilters under consideration here do not admit continuous cofinal maps, they do possess  the key property allowing for the analysis of Tukey reducibility in terms of Rudin-Keisler maps on a front.
We prove in Theorem \ref{thm.canon} that each monotone map from the generic ultrafilter $\mathcal{G}_k$ for $\mathcal{P}(\om^k)/\Fin^{\otimes k}$ into $\mathcal{P}(\om)$ is {\em basic} (see Definition \ref{def.41}) on a filter base for $\mathcal{G}_k$, which implies that it is represented by a finitary function.
This is sufficient for analyzing  Tukey reducibility in terms of Rudin-Keisler maps on fronts.
In the next section, Theorem \ref{thm.canon} will combine  with Theorem \ref{thm.rc} to prove 
 that the initial Tukey structure of nonprincipal ultrafilters below $\mathcal{G}_k$ is exactly a chain of length $k$:  $\mathcal{G}_k>_T\pi_{k-1}(\mathcal{G}_k)>_T\dots>_T\pi_1(\mathcal{G}_k)$.

In Theorem 42 in \cite{Blass/Dobrinen/Raghavan13}, we proved  that each monotone cofinal map from $\mathcal{G}_2$
to some other ultrafilter  is represented by a monotone finitary map which preserves initial segments. 
Here, we extend that result to all  $\mathcal{G}_k$, $k\ge 2$.
Slightly refining
Definition 41 in \cite{Blass/Dobrinen/Raghavan13} and extending it to all $\mathcal{E}_k$, we have the following notion of a canonical cofinal map.

Given  that  $\mathcal{P}(\om^k)/\Fin^{\otimes k}$ is forcing equivalent to $(\mathcal{E}_k,\sse^{\Fin^{\otimes k}})$,
we from now on let  $\mathcal{B}_k$ denote $\mathcal{G}_k\cap\mathcal{E}_k$, where we identify $[\om]^2$ with the upper triangle $\{(i,j):i<j<\om\}$.

\begin{defn}\label{def.41}
Let $2\le k<\om$.
Given $Y\in\mathcal{B}_k$,
a monotone map $g:\mathcal{B}_k| Y\ra\mathcal{P}(\om)$  is {\em basic}
if there is a map $\hat{g}:\mathcal{AR}|Y \ra[\om]^{<\om}$ such that 
\begin{enumerate}
\item
(monotonicity)
For all $s,t\in\mathcal{AR}|Y$, 
$s\sse t\ra\hat{g}(s)\sse\hat{g}(t)$; 
\item
(end-extension preserving)
For $s\sqsubset t$ in $\mathcal{AR}|Y$,
$\hat{g}(s)\sqsubseteq \hat{g}(t)$; 
\item
($\hat{g}$ represents $g$)
For each $V\in\mathcal{B}_k|Y$, 
$g(V)=\bigcup_{n<\om}\hat{g}(r_n(V))$.
\end{enumerate}
\end{defn}

\begin{thm}[Basic monotone maps on $\mathcal{G}_k$]\label{thm.canon}
Let $2\le k<\om$ and $\mathcal{G}_k$ generic for $\mathcal{P}(\om^k)/\Fin^{\otimes k}$ be given.
In $V[\mathcal{G}_k]$,
for each monotone function $g:\mathcal{G}_k\ra\mathcal{P}(\om)$,
there is a $Y\in\mathcal{B}_k$ such that $g\re(\mathcal{B}_k| Y)$ is basic. 
\end{thm}

It follows that every monotone cofinal map $g:\mathcal{G}_k\ra\mathcal{V}$ is represented by a monotone finitary map on the filter base $\mathcal{B}_k|Y$,
for some $Y\in\mathcal{G}_k$.

\begin{proof}
We force with $(\mathcal{E}_k,\sse^{\Fin^{\otimes k}})$, as it is forcing equivalent to 
$\mathcal{P}(\om^k)/\Fin^{\otimes k}$.
Let $\dot{g}$ be an $(\mathcal{E}_k,\sse^{\Fin^{\otimes k}})$-name 
such that $\forces$ ``$\dot{g}:\dot{\mathcal{G}}_k\ra\mathcal{P}(\om)$ is  monotone."
Recall that $\prec$ is a well-ordering on $\om^{\not\,\downarrow\le k}$ with order-type $\om$, and  that  $\lgl\vec{j}_m:m<\om\rgl$ denotes  the $\prec$-increasing  well-ordering of $\om^{\not\,\downarrow \le k}$.
Let $\mathcal{AR}^*$ denote the collection of all trees  of the form $\{Z(\vec{j}_m):m<n\}$, where  $Z\in\mathcal{E}_k$ and $m<\om$.
Note that for those $n<\om$ for which $\vec{j}_{n}$ has length $k$,
$\{Z(\vec{j}_m):m\le n\}$ is a member of $\mathcal{AR}$.

Fix an $A_0\in\mathcal{E}_k$, and let $X_{0}=A_0$.
We now begin the recursive construction of the sequences $(A_n)_{n<\om}$ and $(X_n)_{n<\om}$.
Let $ n\ge 1$ be given, and suppose   we have chosen $X_{n-1},A_{n-1}$.
Let 
$y_n= \{X_{n-1}(\vec{j}_m):m\le n\}$.
Let  $S_n$ denote
 the set of all $z\in\mathcal{AR}^*$ such that
$z\sse y_n$.
Enumerate the members of $S_n$ as $z_n^p$, $p<|S_n|$.
Let $X_n^{-1}=X_{n-1}$ and $A_n^{-1}=A_{n-1}$.
Suppose $p< |S_n|-1$ and we have chosen $X_{n}^{p-1}$ and $A^{p-1}_{n}$.

{\bf  If}
there are $V,A\in\mathcal{E}_k$  with $A\sse^{\Fin^{\otimes k}} V\sse  X_n^{p-1}$ such that
\begin{enumerate}
\item[(i)]
$\widehat{V}\cap y_n= z_n^p$;
\item[(ii)]
$A\forces n-1\not\in \dot{g}(V)$;
\end{enumerate}
{\bf then} 
take  $A_n^p$ and $V_n^p$ to be some such $A$ and $V$.
In this case, $A_n^p\forces n\not\in \dot{g}(V_n^p)$.
Hence, by monotonicity,
$A^p_n\forces n\not\in \dot{g}(V)$ for every $V\sse V^p_n$.
In this case, let 
$X_n^p$ be a member of $\mathcal{E}_k$ such that 
$X_n^p\sse X_n^{p-1}$,
$y_n\sqsubset \widehat{X}_n^p$, and 
whenever $W\sse X^p_n$ such that $\widehat{W}\cap y_n=z^p_n$,
then $W\sse V^p_n$.

{\bf Otherwise}, 
for all $V\sse X_n^{p-1}$ satisfying (i), 
there is no $A\sse^{\Fin^{\otimes k}} V$ which forces $n\not\in\dot{g}(V)$.
Thus, for all $V\sse X_n^{p-1}$ satisfying (i), 
 $V\forces n\in\dot{g}(V)$.
In this case, let $A^p_n=A^{p-1}_n$, $X^p_n=X^{p-1}_n$,
and define $V^p_n$ to be the largest  subset of $X^p_n$ in $\mathcal{E}_k$ such that $\widehat{V}^p_n\cap y_n=z^p_n$.

By this construction, we have that for each $n\ge 1$,
\begin{enumerate}
\item[$(*)$]
$A_n^p$ decides the statement ``$n-1\in\dot{g}(V)$",
\end{enumerate}
for each $V\sse X^p_n$ such that $\widehat{V}\cap y_n=z^p_n$.
Let $A_n=A_n^{|S_n|-1}$ and $X_n=X_n^{|S_n|-1}$.
This ends the recursive  construction of the $A_n$ and $X_n$.
\vskip.1in

Let $Y$ be the set of maximal nodes in the tree $\bigcup_{1\le n<\om} y_n$.
Note that 
$Y$ is a member of $\mathcal{E}_k$.
For $y\in\mathcal{AR}^*$ and $U\in\mathcal{E}_k$,
we let $U/y$ denote the set  $\{U(\vec{i}_m):m<\om$ and $\max U(\vec{i}_m)>\max y\}$.

\begin{claim}\label{claim.Ygood}
For each $V\sse Y$ in $\mathcal{E}_k$ and each $n\ge 1$, 
if $p$ is such that $z^p_n=\widehat{V}\cap y_n$,
then in fact $V\sse V^p_n$.
\end{claim}

\begin{proof}
Let $V\sse Y$ and $n\ge 1$ be given, and let $p$ be such that $z^p_n=V\cap y_n$.
Then
$V/z^p_n= V/y_n\sse Y/y_n\sse X^p_n$,
and every extension of $z^p_n$ into $X^p_n$ is in fact in $V^p_n$.
\end{proof}

Our construction  of $Y$ was geared toward establishing  the following.

\begin{claim}\label{claim.dagger}
Let $V\sse Y$ be in $\mathcal{E}_k$,
and let  $1\le n<\om$ be given.
Let $p$ be the integer such that $\widehat{V}\cap y_n=z^p_n$.
Then 
$$ V\forces  n-1\in \dot{g}(V)  
\Llra 
Y\forces n-1\in\dot{g}(V^p_n).$$
\end{claim}

\begin{proof} 
Given
 $V\sse Y$,
 $1\le n<\om$,
and  $p<|S_n|$ be such that $\hat{V}\cap y_n=z^p_n$.
By $(*)$,
$A^p_n$ decides whether or not $n-1$ is in $\dot{g}(V^p_n)$.
Since $Y\sse^{\Fin^{\otimes k}} A^p_n$,
$Y$ also decides whether or not $n\in\dot{g}(V^p_n)$.
Suppose $Y\forces  n-1\not\in\dot{g}(V^p_n)$.
Since  $V\sse V^p_n$,
by monotonicity of $\dot{g}$,  
we have
$Y\forces n-1\not\in\dot{g}(V)$.
Hence, also $V\forces n-1\not\in \dot{g}(V)$.
Now suppose that $Y\forces n-1\in\dot{g}(V_n^p)$.
Then for all pairs $A\sse^{\Fin^{\otimes k}} V'\sse  X_{n}^{p-1}$ 
satisying (i) and (ii),
we have that
$A\forces n-1\in \dot{g}(V')$.
In particular,
$V\forces n-1 \in\dot{g}(V)$.
\end{proof}

Now we define a finitary monotone  function $\hat{g}:\mathcal{AR}^*|Y \ra [\om]^{<\om}$ which $Y$ forces to represent $\dot{g}$ on the cofinal subset $\mathcal{B}_k|Y$ of $\dot{\mathcal{G}}_k$.
Given $x\in\mathcal{AR}^*|Y$, let $m\ge 1$ be the least integer such that $x\sse y_m$. 
For each $n\le m$, let $p_n$ be the integer such that 
$z^{p_n}_n=x\cap y_n$, and
define
\begin{equation}
\hat{g}(x)=\{n-1: n\le m\mathrm{\ and\ }Y\forces  n-1\in \dot{g}(V^{p_n}_n)\}.
\end{equation}
By definition, $\hat{g}$ is monotone and initial segment preserving.

\begin{claim}\label{claim.psigivesf}
If $Y$ is in $\mathcal{G}_k$, then
$\hat{g}$ represents $\dot{g}$ on $\mathcal{B}_k\re Y$.
\end{claim}

\begin{proof}
Let $V\sse Y$ be in $\mathcal{G}_k$.
Let $n\ge 1$ be given and  let  $p$ such that $z^p_n=\widehat{V}\cap y_n$.
Then Claims  \ref{claim.Ygood}
and \ref{claim.dagger} imply that 
  $V\forces n-1\in \dot{g}(V)$,
 if and only if
$Y\forces n-1\in\dot{g}(V^p_n)$.
This in turn holds if and only if $n-1\in\hat{g}(\widehat{V}\cap y_n)$.
By the definition of $\hat{g}$, we see that 
for each $l<m$, $\hat{g}(\widehat{V}\cap y_l)\sse\hat{g}(\widehat{V}\cap y_m)$.
Thus, $V\forces n-1\in\dot{g}(V)$ if and only if $n-1$ is in $\hat{g}(\widehat{V}\cap y_m)$ for all $m\ge n$.
Therefore, 
$V\forces \dot{g}(V)=\bigcup_{n\ge 1}\hat{g}(\widehat{V}\cap y_n)$.
Thus, the claim holds.
\end{proof}

Finally, we can restrict $\hat{g}$ to have domain $\mathcal{AR}|Y$. 
Note that $\hat{g}$ on this restricted domain retains the property of being monotone and end-extension preserving.
It follows that $Y$ forces  $\hat{g}$ on $\mathcal{AR}$  to represent $g$ on $\mathcal{B}_k|Y$.
To see this, let $V\sse Y$ be in $\mathcal{B}_k$.
Then $\{\hat{g}(r_l(V)):l<\om\}$ is contained in $\{\hat{g}(\widehat{V}\cap  y_n):n<\om\}$, so 
$\bigcup \{\hat{g}(r_l(V)):l<\om\}\sse\bigcup \{\hat{g}(\widehat{V}\cap  y_n):n<\om\}$.
At the same time, for each $n$ there is an $l\ge n$ such that $r_l(V)\contains \widehat{V}\cap y_n$, so monotonicity of $\hat{g}$ implies that $\hat{g}(r_l(V))\contains \hat{g}(\widehat{V}\cap y_n)$.
Thus, $\bigcup \{\hat{g}(r_l(V)):l<\om\}\contains\bigcup \{\hat{g}(\widehat{V}\cap  y_n):n<\om\}$.
Therefore, $Y$ forces that $\hat{g}$ on domain $\mathcal{AR}|Y$  represents $\dot{g}$ on $\dot{\mathcal{B}}_k|Y$, and hence 
$\dot{g}$ is basic on $\dot{\mathcal{B}}_k|Y$.
\end{proof}


\section{The Tukey structure below the generic ultrafilters forced by $\mathcal{P}(\om^k)/\Fin^{\otimes k}$}\label{sec.Tukey}

The recent paper \cite{Blass/Dobrinen/Raghavan13} began the investigation of the Tukey theory of the generic ultrafilter $\mathcal{G}_2$ forced by  $\mathcal{P}(\om\times\om)/\Fin^{\otimes 2}$.
It was well-known that $\mathcal{G}_2$ is the Rudin-Keisler immediate successor of its projected selective ultrafilter $\pi_1(\mathcal{G}_2)$.
In  \cite{Blass/Dobrinen/Raghavan13},  Dobrinen and Raghavan (independently) proved that $\mathcal{G}_2$ is strictly below the maximum Tukey type $([\mathfrak{c}]^{<\om},\sse)$.
Further strengthening that result,  Dobrinen proved that $(\mathcal{G}_2,\contains)\not\ge_T([\om_1]^{<\om}\sse)$, irregardless of the size of the continuum in the generic model.
On the other hand, in Theorem 39 in \cite{Blass/Dobrinen/Raghavan13}, Dobrinen proved that $\mathcal{G}_2>_T\pi_1(\mathcal{G}_2)$.
Thus, we knew that the Tukey type of $\mathcal{G}_2$ is neither maximum nor minimum. 
It was left open what exactly is the structure of the Tukey types of ultrafilters Tukey reducible to $\mathcal{G}_2$.

We solve that open problem here  by showing that  $\mathcal{G}_2$ is the immediate Tukey successor of $\pi_1(\mathcal{G}_2)$, and moreover, each nonprincipal ultrafilter Tukey reducible to $\mathcal{G}_2$ is  Tukey equivalent to either  $\mathcal{G}_2$ or else $\pi_1(\mathcal{G}_2)$.
Thus, the initial Tukey structure of nonprincipal ultrafilters below $\mathcal{U}$ is exactly a chain of order-type 2.
Extending this, we further show that for all $k\ge 2$,
the ultrafilter $\mathcal{G}_k$ 
generic for $\mathcal{P}(\om^k)/\Fin^{\otimes k}$
has initial Tukey structure (of nonprincipal ultrafilters) exactly a chain of size $k$.
We also show that the Rudin-Keisler structures below $\mathcal{G}_k$ is 
exactly a chain of size $k$.
Thus, the Tukey structure below $\mathcal{G}_k$ is analogous to  the Rudin-Keisler structure below $\mathcal{G}_k$, even though each Tukey equivalence class contains many Rudin-Keisler equivalence classes.

Let $k\ge 2$.
As in the previous section, we let 
Let  $\mathcal{G}_k$  be a generic ultrafilter forced by 
 $\mathcal{P}(\om^k)/\Fin^{\otimes k}$, and let $\mathcal{B}_k$ denote $\mathcal{G}_k\cap\mathcal{E}_k$, where we are identifying $[\om]^k$ with the collection of strictly increasing sequences of natural numbers of length $k$.
Then $\mathcal{B}_k$  is a generic filter for 
 $(\mathcal{E}_k,\sse^{\Fin^{\otimes k}})$, and  $\mathcal{B}_k$ is cofinal in $\mathcal{G}_k$.

We begin by showing that each $\mathcal{G}_k$ has at least $k$-many distinct Tukey types of nonprincipal ultrafilters below it, forming a chain.
The proof of the next proposition is very similar to the proof of  Proposition 39 in \cite{Blass/Dobrinen/Raghavan13}, which showed that $\mathcal{G}_2>_T \pi_1(\mathcal{G}_2)$.

\begin{prop}\label{prop.1<T2}
Let $k\ge 2$ and $\mathcal{G}_k$ be generic for $\mathcal{P}(\om^k)/\Fin^{\otimes k}$.
 Then in $V[\mathcal{G}_k]$,
for each $l<k$, 
 $\pi_l(\mathcal{G}_k)<_T \pi_{l+1}(\mathcal{G}_k)$.
\end{prop}

\begin{proof}
Since the map $\pi_l:\pi_{l+1}''\bW_k\ra\pi_l''\bW_k$ witnesses that $\pi_l(\mathcal{B}_k)\le_{RK}\pi_{l+1}(\mathcal{B}_k)$,
it follows that 
 $\pi_l(\mathcal{B}_k)\le_T\pi_{l+1}(\mathcal{B}_k)$.
Thus, it remains only to show that these are not Tukey equivalent.
Let $\dot{g}:\pi_l(\mathcal{G}_k)\ra\pi_{l+1}(\mathcal{G}_k)$ be a $(\mathcal{E}_k,\sse^{\Fin^{\otimes k}})$-name for a monotone map.
Without loss of generality, we may identify $\pi_{l+1}''\bW_k$ with $\om$.

Noting that $\pi_l(\mathcal{E}_k):=\{\pi_l(X):X\in\mathcal{E}_k\}$ is isomorphic to $\mathcal{E}_l$, and that $\pi_l(\mathcal{E}_k)$ is regularly embedded into $\mathcal{E}_k$, 
it follows by a slight modification of the proof of Theorem \ref{thm.canon} that 
there is some $A\in\mathcal{B}_k$ such that $A$ forces that $\dot{g}\re\pi_l(\mathcal{B}_k|A)$ is basic.
Thus, in $V[\mathcal{G}_k]$,
$g$ is represented by finitary monotone initial segment preserving map $\hat{g}$ defined on $\mathcal{AR}|Y$.
Letting $f$ denote the map on $\{\pi_l(X):X\in\mathcal{E}_k|A\}$ determined by $\hat{g}$,
we see that $f$ is actually in the ground model since
$(\mathcal{E}_k,\sse^{\Fin^{\otimes k}})$ is a $\sigma$-closed forcing.

Let $X\in\mathcal{B}_k|A$ be given.
If there is a $Y\sse X$  in $\mathcal{B}_k$ such that $f(\pi_l(Y))\cap \pi_{l+1}(Y)$ does not contain a member of $\pi_{l+1}(\mathcal{E}_k)$,
then $Y$ forces that $f(\pi_l(\dot{\mathcal{B}}_k))\not\sse\pi_{l+1}(\dot{\mathcal{B}}_k)$.
Otherwise, 
(a) for all $Y\sse X$ in $\mathcal{B}_k$, $f(\pi_l(Y))\cap \pi_{l+1}(Y)$ is a member of $\pi_{l+1}(\mathcal{E}_k)$.

If there is a $Y\sse X$ in $\mathcal{B}_k$ such that for all $Z\sse Y$ in $\mathcal{B}_k$, $f(\pi_l(Z))\not\sse \pi_{l+1}(Y)$,
then $Y$ forces that $f\re\pi_l(\dot{\mathcal{B}}_k)$ is not cofinal into $\pi_{l+1}(\dot{\mathcal{B}}_k)$.
Otherwise,
(b) for all $Y\sse X$ in $\mathcal{B}_k$, there is a $Z\sse Y$ in $\mathcal{B}_k$ such that $f(\pi_l(Z))\sse \pi_{l+1}(Y)$.

Now we are in the final case that (a) and (b) hold.
Fix $Y,W\sse X$ such that $\pi_l(Y)=\pi_l(W)$ but $\pi_{l+1}(Y)\cap \pi_{l+1}(W)=\emptyset$.
Take $Y'\sse Y$ such that $f(\pi_l(Y'))\sse\pi_{l+1}(Y)$.
Take $W'\sse W$ such that $\pi_l(W')\sse\pi_l(Y')$;
then take $W''\sse W$ such that $f(\pi_l(W''))\sse \pi_{l+1}(W')$.
Since  $\pi_l(W'')\sse\pi_l(Y')$ and $f$ is monotone, we have that
 $f(\pi_l(W''))\sse f(\pi_l(Y'))$.
Thus, $f(\pi_l(W''))\sse \pi_{l+1}(Y)$.
On the other hand,  $f(\pi_l(W''))\sse \pi_{l+1}(W')$, which is contained in $\pi_{l+1}(W)$.
Hence, $f(\pi_l(W''))\sse \pi_{l+1}(Y\cap W)$, which is empty.
Thus, $W''$ forces $f(\pi_l(Z))$ to be the emptyset, for each $Z\sse W''$,
so $W''$ forces $f$ not to be a cofinal map.
\end{proof}

Applying Theorems \ref{thm.rc} and \ref{thm.canon}, we shall prove the main theorem of this paper.

\begin{thm}\label{thm.main}
Let  $k\ge 2$,  and let $\mathcal{G}_k$ be generic for the forcing $\mathcal{P}(\om^k)/\Fin^{\otimes k}$.
If $\mathcal{V}\le_T\mathcal{G}_k$ and $\mathcal{V}$ is nonprincipal,
then  $\mathcal{V}\equiv_T\pi_l(\mathcal{G}_k)$,
 for some $l\le k$.
\end{thm}

\begin{proof}
Let $\mathcal{G}_k$ be a $\mathcal{P}(\om^k)/\Fin^{\otimes k}$ generic ultrafilter on $\om^k$,
and let $\mathcal{B}$ denote $\mathcal{B}_k$.
Let $\mathcal{V}$ be a nonprincipal  ultrafilter on base set $\om$ which is Tukey reducible to $\mathcal{G}_k$.
Then there is a monotone cofinal map
 $g:\mathcal{G}_k\ra\mathcal{V}$  witnessing that $\mathcal{V}$ is Tukey reducible to $\mathcal{G}_k$.
By Theorem \ref{thm.canon}, there is an $A\in\mathcal{B}$ 
such that  $g$ on $\mathcal{B}|A$ is basic, represented by 
 a finitary, monotone, end-extension preserving  map 
$\hat{g}:\mathcal{AR}|A\ra [\om]^{<\om}$.

For each $X\in\mathcal{B}|A$,  let $a_X=r_{n}(X)$ where $n$ is   least   such that $\hat{g}(r_{n}(X))\ne\emptyset$.
Let $\mathcal{F}=\{a_X:X\in\mathcal{B}|A\}$.
Note that $\mathcal{F}$ is a front on $\mathcal{B}|A$.
For  $X\in\mathcal{B}|A$, recall that $\mathcal{F}|X$ denotes $\{a\in\mathcal{F}:a\le_{\fin} A\}$.
We let $\lgl \mathcal{B}\re\mathcal{F}\rgl$ denote  the filter on the base set $\mathcal{F}$ generated by the collection of sets $\mathcal{F}|X$, $X\in\mathcal{B}|A$.
Define $f:\mathcal{F}\ra\om$ by $f(a)=\min \hat{g}(a)$.
By genericity of $\mathcal{G}_k$ and arguments for Facts 5.3 and 5.4 and Proposition 5.5 in \cite{Dobrinen/Todorcevic14},
it follows that
 $\mathcal{V}=f(\lgl \mathcal{B}\re\mathcal{F}\rgl)$;
that is, $\mathcal{V}$ is the ultrafilter which is the Rudin-Keisler image via $f$ of the filter $\lgl \mathcal{B}\re\mathcal{F}\rgl$.

By Theorem \ref{thm.rc} and genericity of $\mathcal{G}_k$, there is a 
$B\in\mathcal{B}|A$ such that $f\re\mathcal{F}|B$ is canonical, represented by an inner Nash-Williams function $\vp$.
Recall from the proof of Theorem \ref{thm.rc} that $\vp$ is a projection function, where 
$\vp(a)=\bigcup\{\pi_{r_{i}(a)}(a(i)): i< |a|\}$, for $a\in \mathcal{F}|B$.

For  $l\le k$ and $X\in\mathcal{E}_k|B$,
we say that $(*)_l(X)$ holds if and only if
for each $Y\le X$, for each $Z\le Y$, there is a $Z'\le Z$ such that $\pi_l(Z')\sse\vp(\mathcal{F}|Y)$
and, if $l<k$, then also $\pi_{l+1}(X)\cap \vp(\mathcal{F}|X)=\emptyset$.

\begin{claim}\label{claim.forces}
If $(*)_l(X)$ and $\neg(*)_{l+1}(X)$,
then $X$ forces $\vp(\mathcal{G}_k|\mathcal{F})\equiv_T\pi_l(\mathcal{G}_k)$.
\end{claim}

\begin{proof}
Let  $l\le k$ be given and suppose that $(*)_l(X)$  holds, and if $l<k$, then also 
$\neg(*)_{l+1}(X)$.
By definition of $\vp$,
we know that  for each $Y\le X$, $\vp(\mathcal{F}|Y)\sse\bigcup_{i\le k}\pi_i(Y)$.
By $\neg(*)_{l+1}(X)$,
we have that $\vp(\mathcal{F}|Y)$ must actually be contained in $\bigcup_{i\le l}\pi_i(Y)$.
$(*)_l(X)$ implies that $X$ forces that for each $Y\le X$ in $\dot{\mathcal{G}}_k$, there is a $Z'\le Y$ in $\dot{\mathcal{G}}_k$ such that 
$\pi_l(Z')\sse\vp(\mathcal{F}|Y)$.
Then $\pi_l(\mathcal{G}_k)$ is actually equal to 
the filter generated by the sets $(\bigcup\vp(\mathcal{F}|Y))\cap\pi_l(\mathcal{E}_k)$,  $Y\in\mathcal{G}_k$, since they are cofinal in each other.
Moreover, the filter generated by the sets $(\bigcup\vp(\mathcal{F}|Y))\cap\pi_l(\mathcal{E}_k)$,  $Y\in\mathcal{G}_k$, is Tukey equivalent to $\vp(\mathcal{G}_k|\mathcal{F})$, as can be seen by the map
$\vp(\mathcal{F}|Y)\mapsto (\bigcup\vp(\mathcal{F}|Y))\cap\pi_l(\mathcal{E}_k)$, 
which is easily seen to be both  cofinal and Tukey.
\end{proof}

\begin{claim}\label{claim.existsl}
For each $W\in\mathcal{E}_k|B$, there is an $X\le W$ and an $l\le k$ such that 
$(*)_l(X)$ holds.
\end{claim}

\begin{proof}
Let $W\in\mathcal{E}_k|B$ be given. 
For all pairs $j\le l\le k$, define
\begin{equation}
\mathcal{H}^j_l=\{a\in\mathcal{F}|W:\exists n<|a|(n\in N^k_{j}\wedge\vp_{r_n(a)}=\pi_l)\}.
\end{equation}
Take $X\le W$ homogeneous for $\mathcal{H}^j_l$ for all $j\le l\le k$.
Let $l\le k$ be maximal such that, for some $j\le l$,
$\mathcal{F}|X\sse\mathcal{H}^j_l$.
We point out that $\mathcal{F}|W\sse\mathcal{H}^0_0$, so such an $l\le k$ exists.
We claim that $(*)_l(X)$ holds.

Note that, if $l<k$, then for all $l<l'\le k$,
$(\mathcal{F}|X)\cap\mathcal{H}^j_{l'}=\emptyset$, whenever $j\le l'$.
Thus, for each $a\in\mathcal{F}|X$,
there is no $n<|a|$ for which $\vp_{r_n(a)}=\pi_{l'}$.
Therefore, for each $a\in\mathcal{F}|X$, $\vp(a)\sse\bigcup_{i\le l}\pi_i(X)$.

Now let $j\le l$ such that $\mathcal{F}|X\sse\mathcal{H}^j_l$, and let 
 $Z\le Y\le X$ be given.
If there is a $C\in\mathcal{E}_j$ such that 
$C\sse\{\pi_j(a(n)):a\in\mathcal{F}|Z$, $n<|a|$, $n\in N^k_{j}$, and $\vp_{r_n(a)}=\pi_l\}$,
then there is a $Z'\le Z$ such that $\pi_j(Z')\sse C$.
It follows that $\pi_l(Z')\sse C\sse\vp(\mathcal{F}|Z)$.

Such a $C\in\mathcal{E}_j$ must exist, for if there is none,
then there is a $C'\in\mathcal{E}_j$ such that 
$C'\cap\{\pi_j(a(n)):a\in\mathcal{F}|Z$, $n<|a|$, $n\in N^k_j=\emptyset$.
In this case there is a $Z'\le Z$ such that $\pi_j(Z')\sse C'$.
But then $\pi_l(Z')\cap \vp(\mathcal{F}|Z)=\emptyset$, contradicting that $\mathcal{F}|X\sse\mathcal{H}^j_l$.
Thus, there is a $Z'\le Z$ such that $\pi_l(Z')\sse\vp(\mathcal{F}|Z)$, which in turn is contained in $\vp(\mathcal{F}|Y)$.
Therefore, $(*)_l(X)$ holds.
\end{proof}

Thus, by Claims \ref{claim.forces} and \ref{claim.existsl},
it is dense in $\mathcal{E}_k$  to force that $\vp(\mathcal{G}_k|\mathcal{F})\equiv_T\pi_l(\mathcal{G}_k)$ for some $l\le k$.
\end{proof}

We finish by showing that each ultrafilter Rudin-Keisler reducible to $\mathcal{G}_k$ is actually Rudin-Keisler equivalent to $\pi_l(\mathcal{G}_k)$ for some $l\le k$.

\begin{thm}\label{thm.rkinitstructure}
Let  $k\ge 2$,  and let $\mathcal{G}_k$ be generic for the forcing $\mathcal{P}(\om^k)/\Fin^{\otimes k}$.
If $\mathcal{V}\le_{RK}\mathcal{G}$ and $\mathcal{V}$ is nonprincipal,
then  $\mathcal{V}\equiv_{RK}\pi_l(\mathcal{G}_k)$,
 for some $l\le k$.
\end{thm}

\begin{proof}
Let $\mathcal{V}\le_{RK}\mathcal{G}_k$.
Note that $\mathcal{G}_k$ is isomorphic to the ultrafilter  $\mathcal{G}_k\re \mathcal{AR}_1$ having base set $\mathcal{AR}_1$.
Thus, there is a function $h:\mathcal{AR}_1\ra \om$ which witnesses that 
$h(\mathcal{G}_k\re\mathcal{AR}_1)=\mathcal{V}$.
Such an $h$ induces an equivalence relation on $\mathcal{AR}_1$.
Applying Theorem \ref{thm.rc},
there is an $A\in\mathcal{G}_k$ such that $h\re\mathcal{AR}_1|A$ is represented by an irreducible map on $\mathcal{AR}_1|A$.
The only irreducible maps on first approximations are the projection maps $\pi_l$, $l\le k$.
Thus, 
$h(\mathcal{G}_k\re\mathcal{AR}_1)$ must be  exactly $\pi_l(\mathcal{G}\re \mathcal{AR}_1)$ for some $l\le k$.
Hence, $\mathcal{V}$ is isomorphic to $\pi_l(\mathcal{G}_k)$, for some $l\le k$.
\end{proof}

Thus, the initial Tukey structure mirrors the initial Rudin-Keisler structure, even though each Tukey type contains many Rudin-Keisler isomorphism classes.


\section{Further directions}

Noticing that $[\om]^k$ is really a uniform barrier on $\om$ of rank $k$, we point out that our method of constructing  Ellentuck spaces of dimension $k$  can be extended transfinitely using uniform barriers of any countable rank.
The members of the spaces will not simply be restrictions of the barrier to infinite sets, but rather will require the use of auxiliary structures in the same vein as were used in \cite{Dobrinen/Todorcevic15} to construct the spaces $\mathcal{R}_{\al}$ for $\om\le \al<\om_1$.

In \cite{Dobrinen/Mijares/Trujillo14},
Dobrinen, Mijares, and Trujillo presented a template for constructing new topological Ramsey spaces which have on level 1 the Ellentuck space, and on level 2 some finite product of finite structures from a \Fraisse\ class of ordered relational structures with the Ramsey property.
They showed that any finite Boolean algebra appears as the initial Tukey structure of a p-point associated with some  space constructed by that method.
Moreover, that template also constructs topological Ramsey spaces for which the maximal filter is essentially a Fubini product of p-points, and which has initial Tukey structure consisting of  all Fubini iterates of a collection of p-points which is Tukey ordered as $([\om]^{<\om},\sse)$.
(See for instance the space $\mathcal{H}^{\om}$ in Example 25 ub \cite{Dobrinen/Mijares/Trujillo14}.)

\begin{problem}
Construct  topological Ramsey spaces with associated ultrafilters which are neither p-points nor Fubini products of p-points, but which have initial Tukey structures which are not simply chains.
\end{problem}

We conclude with a conjecture about what is actually necessary to prove the Abstract Nash-Williams Theorem for general topological Ramsey spaces.
In our proof of the Ramsey-classification theorem for equivalence relations on fronts, the Abstract Nash-Williams Theorem was sufficient; we did not need the full strength of the Abstract Ellentuck Theorem. 
The fact that (in earlier versions of this paper), we proved the Abstract Nash-Williams Theorem for the spaces $\mathcal{E}_k$ without using \bf A.3 \rm (b) leads to the following conjecture.

\begin{conjecture}
Let  $(\mathcal{R},\le,r)$ be a space for which $\mathcal{R}$ is a closed subspace of $\mathcal{AR}^{\om}$ and axioms \bf A.1 \rm through \bf A.4 \rm minus \bf A.3 \rm (b) hold.
Then  the Abstract Nash-Williams Theorem holds.
\end{conjecture}

\bibliographystyle{amsplain}
\bibliography{references}

\end{document}